\documentclass[10pt, a4paper,reqno]{amsart}

\usepackage{mathtools}
\usepackage{amsmath,amsthm, amsfonts, amsrefs, amssymb}
\mathtoolsset{showonlyrefs}

\usepackage[colorlinks, linkcolor = black, citecolor = black, filecolor = black, urlcolor = blue]{hyperref}
\usepackage{graphicx}
\usepackage{bbm}
\usepackage{tikz}
\usepackage{tikz-cd}
\usepackage{hyperref}
\usepackage{todonotes}
\usepackage{mathrsfs}
\usepackage[margin=2cm]{geometry}
\usepackage{enumitem}
\usepackage{cancel}

\theoremstyle{plain}
\newtheorem{thm}{Theorem}[section]
\newtheorem{lemma}[thm]{Lemma}
\newtheorem{col}[thm]{Corollary}
\newtheorem{prop}[thm]{Proposition}

\newtheorem*{thm*}{Theorem}

\theoremstyle{definition}
\newtheorem{defn}[thm]{Definition}

\newtheorem{remark}[thm]{Remark}

\numberwithin{equation}{section}

\newcommand{\nc}{\newcommand}
\nc{\sgn}{\mathrm{sgn}}
\nc{\nothing}{\varnothing}
\nc{\Ab}{\mathrm{Ab}}
\nc{\Graph}{\mathbf{Graph}}
\nc{\FinSet}{\mathbf{FinSet}}
\nc{\Vect}{\mathbf{Vect}}
\nc{\Top}{\mathbf{Top}}
\nc{\Cat}{\mathbf{CAT}}
\nc{\Set}{\mathbf{Set}}
\nc{\Rel}{\mathbf{Rel}}
\nc{\Grp}{\mathbf{Grp}}
\nc{\AbGrp}{\mathbf{AbGrp}}
\nc{\cT}{\mathcal T}
\nc{\cG}{\mathcal G}
\nc{\cF}{\mathcal F}
\nc{\cK}{\mathcal{K}}
\nc{\cE}{\mathcal E}
\nc{\cP}{\mathcal P}
\nc{\cM}{\mathcal M}
\nc{\cC}{\mathcal C}
\nc{\cB}{\mathcal B}
\nc{\cS}{\mathcal S}
\nc{\cD}{\mathcal D}
\nc{\cA}{\mathcal A}
\nc{\PP}{\mathbb P}
\nc{\cU}{\mathcal U}
\nc{\Mod}{\operatorname{Mod}}
\nc{\Aut}{\operatorname{Aut}}
\nc{\del}{\partial}
\nc{\inter}{\mathrm{o}}
\nc{\close}[1]{\overline{#1}}
\nc{\pderiv}[2]{\frac{\partial #1}{\partial #2}}
\nc{\tr}{\operatorname{tr}}
\nc{\dd}{\mathrm{d}}
\renewcommand{\epsilon}{\varepsilon}

\nc{\RP}{\R\mathrm{P}}

\newcommand{\R}{\mathbb{R}}
\renewcommand{\P}{\mathbb{P}}
\newcommand{\E}{\mathbb{E}}
\newcommand{\N}{\mathbb{N}}
\newcommand{\Z}{\mathbb{Z}}

\newcommand{\T}{\mathbb{T}}

\newcommand{\id}{\text{id}}
\newcommand{\inv}{^{-1}}

\nc{\ev}{\mathrm{ev}}
\nc{\Nat}{\mathrm{Nat}}

\newcommand{\ra}{\rightarrow}

\newcommand{\diver}{\mathrm{div}}
\nc{\vect}[1]{\mathbf{#1}} 
\nc{\im}{\mathrm{im}}
\nc{\weakra}{\rightharpoonup}
\nc{\Cof}{\mathrm{Cof}}
\nc{\innprod}[2]{\left\langle #1, #2 \right\rangle}
\nc{\norm}[1]{\left\| #1 \right\|}
\nc{\abs}[1]{\left\lvert #1 \right\rvert}
\nc{\e}[1]{ \mathbb E \left[ #1 \right] }
\renewcommand{\d}{\mathrm{d}}
\newcommand{\loc}{\mathrm{loc}}
\nc{\vp}{\varphi}
\nc{\supp}{\mathrm{supp}}
\nc{\eint}[1]{e^{-C #1}\int_{\T^d}}
\nc{\tint}{\int_{\T^d}}
\nc{\dist}[2]{\mathrm{d}_{\T^d}(#1,#2)}
\nc{\Eint}[2]{e^{-C #1}\int_{#2}}
\nc{\eps}{\epsilon}
\nc{\ukl}{u_{\kappa,l}}

\newcommand{\lltriangle}{\rotatebox[origin=c]{360}{\tikz{\draw[thick](0,0)--(0.25,0)--(0,0.25)--cycle;}}}

\allowdisplaybreaks

\usepackage[dvipsnames]{xcolor}
\title[Parabolic-hyperbolic splitting in support propagation for SPMEs]{Parabolic-hyperbolic splitting in support propagation for stochastic  porous media equations}
\author{Max Sauerbrey}
\address{Max Planck Institute for Mathematics in the Sciences\\
Inselstr. 22 \\ 04103 Leipzig \\ Germany.} \email{maxsauerbrey97@gmail.com}
\author{Joshua Utley}
\address{Department Mathematik\\ Friedrich--Alexander--Universit\"at Erlangen--N\"urnberg\\ Cauerstra{\ss}e 11\\
91058 Erlangen\\ Germany}
\email{utley1999@gmail.com}
\date{\today}


\begin{document}

\begin{abstract} 
    We develop a {randomly-localized} energy method to derive qualitative results on the support propagation of stochastic porous media equations with linear conservative noise.
   Unlike in previous works, where energies are localized by weighting them with a spatial bump function, we weight them with profiles which are solutions to stochastic transport equations. This modulates out the support propagation due to the conservative noise term, and energy arguments---which otherwise fail in this setting---are again applicable.
    As a result, finite speed of propagation along with sufficient and necessary conditions on the existence of waiting time phenomena (modulo stochastic transport) are proven. These methods and results demonstrate, on short time scales, that the support propagation may be disintegrated into two independent parts, one due to the evolution of the porous media equation and one due to stochastic transport. 
\end{abstract}

\maketitle

\section{Introduction}\label{S:Introduction}

In this work we examine qualitative properties of the support evolution for non-negative solutions to stochastic porous media equations of the form 
 \begin{equation}\label{Eq_SPME_intro}\tag{sPME}
 	\begin{cases}
 		\d u \,=\, \Delta (u^m) \,\d t \,-\, \diver( u  \circ \d W ),\\
 		u(0) \,=\, u_0,\end{cases}
 \end{equation}
where $m\in (1,\infty)$ and $W$ is a spatially colored Wiener process. 
The forward support propagation of porous media equations with multiplicative noise was studied in works by Barbu and R\"ockner \cite{Barbu_Rock_FSOP},  Gess \cite{gess_fsop} and Fischer and Gr\"un \cite{Fischer_Grun_FSOP}. The stochastic version of the well-established localized energy method for deterministic PDEs (see, e.g., \cite{Shishkov_Hulshof_FSOP,AnsiniGiacomelli2004,DalPasso_Giacomelli_Grun_PISA,GS,Grun2001b}) proposed in the latter even yields sufficient conditions for the occurrence of waiting time phenomena, and is also applicable to SPDEs with nonlinear conservative noise \cite{GK_fsop3,GSU_SPME_Superlinear}. However, it does not yield any insight into the support propagation of \eqref{Eq_SPME_intro}. This is due to the presence of the  non-degenerate hyperbolic term, a scenario which needed to be excluded also in prior treatments of deterministic PDEs \cite{GS}. 
To alleviate this issue, we develop a variant of the aforementioned stochastic energy method, in which the spatial localization moves along the (stochastic) flow of diffeomorphisms induced by 
\begin{align}\label{eqn_flowww}
	\d \phi_t = \circ \, \d W( t,\phi_t ), \quad \phi_0 =x,\quad x\in\T^d,
\end{align}
and thereby absorbs the obstructing contributions of the hyperbolic term. Based on this idea, we also provide a randomly localized version of approach of Chipot and Sideris \cite{ChipotSideris1985}, which was employed to obtain necessary conditions for the occurrence of a waiting time phenomenon for deterministic porous media equations. Regarding \eqref{Eq_SPME_intro}, this allows us to prove the following for its unique, non-negative kinetic solution $u$: 

\begin{enumerate}[label=(\Roman*)]
	\item $u$ has finite speed of propagation.
	\item\label{IA} If the profile of $u_0$ is sufficiently flat near the boundary of its support, then, $\P$-a.s., for short times, the support propagation of $u$ is given by the transport along $\phi$, i.e., there is a waiting time phenomenon modulo $\phi$. 
	\item\label{IB} If the profile of $u_0$ is sufficiently steep near the boundary of its support, then, $\P$-a.s., for short times, the support of $u$ propagates further than is given by pure transport along $\phi$, i.e., there is no waiting time phenomenon modulo $\phi$. 
\end{enumerate} 
The above are the first results on the support propagation of not just \eqref{Eq_SPME_intro}, but also generally for degenerate parabolic (S)PDEs with non-degenerate hyperbolic terms. 
We remark furthermore that the terms ``sufficiently flat" (cf.\ Corollary \ref{col:lower}) and ``sufficiently steep" (cf.\ Corollary \ref{col:upper}) are sharp in scaling up to a logarithmic correction, so that \ref{IA} and \ref{IB} establish a near-dichotomy for power-law-type initial data (s. the discussion below Corollary \ref{col:upper}). This seems to be the first instance of such a precise description of the forward support propagation of SPDEs in the literature, as the only other result on necessary conditions for the occurrence of a waiting time phenomenon we are aware of from \cite{GrillmeierDissertation} for porous media equations with multiplicative noise does not align well with the sufficient conditions given in \cite{Fischer_Grun_FSOP}. 

The rest of the paper is structured as follows: In the remainder of this section, we motivate and give a high-level overview of the proposed technique to treat \eqref{Eq_SPME_intro}, after which we refer to related literature and introduce our notation conventions. Section \ref{SS:Assumptions_Main_Results} contains a statement of the assumptions and rigorous statements of the main results, recalling in particular the kinetic solution framework of Fehrman and Gess \cite{Fehrman_Gess_ARMA}. In Section \ref{S:flows}, we review  relevant properties of the stochastic flow of diffeomorphisms $\phi$ and its connection to stochastic transport equations. Section \ref{SS:prop_kinetic} is dedicated to the proofs of randomly localized identities and estimates for \eqref{Eq_SPME_intro}, which are essential  to prove the main results. Section \ref{SS:proof_upper} contains the proof of finite speed of propagation as well as the proof of the sufficient conditions for the occurrence of a waiting time phenomenon. Section \ref{SS:proof_lower} contains the proof of the necessary condition for the occurrence of a waiting time phenomenon. Other technical results, including an It\^o product rule for SPDEs, may be found in the appendixes.

\subsection{The randomly-localized energy method}\label{ss:method}

Let us  say a few words about the need for, and the execution of, the randomly-localized energy method that we develop. We start with the question of finite speed of propagation and sufficient conditions for the occurrence of a waiting time phenomenon. On the one hand, the first results in this direction for stochastic porous media equations \cite{Barbu_Rock_FSOP,gess_fsop} rely on a random rescaling of the equation, which transforms the SPDE into a random PDE.  In both cases, the linear multiplicative structure of the noise is crucial, obstructing an adaption to \eqref{Eq_SPME_intro}. On the other hand, while  the more general, energy-based methods of \cite{Fischer_Grun_FSOP} are flexible enough to be applied to a variety of equations, such as fourth-order SPDEs \cite{GK_fsop3} and porous media equations with nonlinear conservative noise \cite{GSU_SPME_Superlinear}, they fail for \eqref{Eq_SPME_intro} as well.

It is instructive to see exactly how the energy method fails, as it motivates our approach. The main idea of the stochastic energy method is to show that the following stopping times (called waiting times)
\[
    T_w(U) = \inf\left\{ t>0, \quad \int_U u(t,x)\, \d  x > 0\right\}
\]
are $\P$-a.s.\ positive for an open set $U$ which is disjoint from the initial support. Throughout the manuscript, we default to the choice of $U$ as an open ball as a matter of simplicity. To show that $T_w(U)$ is $\PP$-a.s. positive, one first derives estimates in expectation on functionals of the form 
\begin{align}\label{eqn_AG}
    \tint \eta u^2, \quad \eta\in C^2(\T^d;[0,1]).
\end{align}
These localized energy estimates are then combined with the Gagliardo--Nirenberg inequality and an appropriate iteration argument, referred to as filtering, to show that 
\[
    \PP (T_w(U)= 0) = 0.
\]
For a more-detailed exposition on the method, we refer one to its original use in \cite{Fischer_Grun_FSOP} or also to the exposition and usage in \cite{GSU_SPME_Superlinear}. An indispensable ingredient for the application of this method is that powers of $u$ that appear on the right-hand side of the local energy estimate are {greater than two}. However, if $u$ is taken to be a solution to \eqref{Eq_SPME_intro} and one applies the It\^o formula to the above functional, it quickly becomes apparent that this is not the case. Even worse, there will always be  terms of the form
\[
    \int_0^t\int \nabla \eta\cdot u\nabla u
\]
on the right-hand side, which also appear when deriving estimates of this type for solutions to, for instance, the heat equation. Since the latter has  infinite speed of support propagation, a proof of finite speed of propagation based on such an expansion is not feasible.

To alleviate this, we propose to take $\eta$ itself as a solution to an SPDE, namely to the stochastic transport equation 
\begin{align} \label{eqn_tpe} \tag{sTE}
  \begin{cases}   \d \eta = -\nabla  \eta \circ \d W =  -\sum_{k\in \N}  
    \psi_k \nabla \eta \circ \d \beta_k, \\
    \eta(0) =\eta_0.
    \end{cases}
\end{align}
In fact, if $u$ were to solve \eqref{Eq_SPME_intro} without the porous media part, i.e., if $\d u = -\diver( u \circ \d W) $, then the well-known duality of conservation and transport equations dictates that $\d \int \eta u =0$, thereby annihilating the convective effects on the localized mass. Our key observation is as follows: When the porous media operator is present and localized energies such as \eqref{eqn_AG} are considered, the stochastic convection is, up to leading order, still absorbed by taking $\eta$ as a solution to \eqref{eqn_tpe}.
The rigorous implementation of this idea requires us to overcome some technical challenges:  First, we require an It\^o formula for products of solutions to SPDEs in order to compute the evolution of $\int \eta u^2$. Secondly, it is important for the iteration arguments to know exactly what the support propagation of $\eta$ looks like. In fact, we need to consider an infinite family of solutions to \eqref{eqn_tpe} parameterized by a suitable family of initial data. For the former, we prove an appropriate It\^o product formula, which is rather general and perhaps useful in other contexts (Appendix \ref{app:ito}). Concerning the support of $\eta$, we use the well-known relation of \eqref{eqn_tpe} to the stochastic flow of diffeomorphisms induced by \eqref{eqn_flowww} that was discovered by Kunita \cite{Kunita,kunita_book}: If $W$ is sufficiently regular in space,
\eqref{eqn_flowww} generates a stochastic flow of $C^2$-diffeomorphisms $\{\phi_t(x)\}_{x\in\T^d}$, which yields the representation formula
\[
    \eta(t) = \eta_0\circ \phi_t\inv
\]
for the solutions $\eta$ to \eqref{eqn_tpe}. The latter gives bounds on infinitely many $\eta$ given that $\eta_0$ is appropriately chosen and the flow $\phi$ is controlled. More than that, it gives a description of $\supp(\eta(t))$ in terms of $\supp(\eta_0)$ allowing us to interpret the quantities \eqref{eqn_AG} in terms of $\phi$.
With this at hand, we can show by a suitable version of the filtering technique that  
\[
    \PP( T^\phi_w(U) = 0) = 0,
\]
where $T^\phi_w(U)$---which we call the parabolic waiting time, since it isolates the parabolic spreading behavior---is the first time that mass enters $\phi_t(U)$.
We will find that the latter holds for open balls $U$ that have a positive distance from $\supp(u_0)$, which implies finite speed of propagation, since the flow $\phi$ is continuous. If $U$ has zero distance to the initial support (but is still disjoint), then we find that $T_w^\phi(U)$ is $\PP$-a.s. positive if we impose a flatness condition on the profile of the initial data. 

As it turns out, this perspective of random localization is capable of more. Examining in addition the randomly localized mass 
\[
    \tint \eta(t,x) u(t,x) \, \d x,
\]
we find that we are able to derive differential inequalities as in the classical argument in \cite{ChipotSideris1985}, which established upper bounds on waiting times for solutions to deterministic porous media equations. The structure is such that we may adapt their arguments up to a positive stopping time $\sigma$, using that the flow is sufficiently close to the identity map for small times. This leads to a condition on $u_0$ such that $\sigma \wedge T_w^\phi(U) = 0$ and therefore $ T_w^\phi(U) = 0$, $\PP$-a.s., for sets that share a subset of their boundary with $\supp(u_0)$. In other words, the free boundary immediately moves forward, even modulo the flow $\phi$. 

\subsection{Other related literature}

Because the porous media equation is a prototypical example of a nonlinear parabolic equation, the collection of literature regarding stochastic porous media equations is vast. Going back to \cite{pardoux, krylov_rozovski}, stochastic variants of monotone operator theory have been tailored towards porous media equations in \cite{wang1, wang2} and further improved via Yosida approximations in \cite{BPR_non_neg, BPR_strong_sol_PME}, where also non-negativity of solutions was obtained. Unresolved  questions regarding higher regularity of solutions were subsequently addressed in \cite{BPR_criticality,Gess_strong}. A summary of these developments is the content of the monograph \cite{spme_book} and we refer to \cite{LR} regarding further improvements of the variational framework for SPDEs obtained in \cite{liu_monotone_article,LR_article, LR_JDE}. We also mention the recent contributions  \cite{AV_montone, RSZ} relaxing the local monotonicity condition which is commonly imposed. Stochastic porous media equations with hyperbolic noise terms have gained more attention as well for their connection to fluctuations about a limiting profile for zero-range processes \cite{Gess_Heydecker,FG_inventiones,zimmer}. 

Regarding energy methods for degenerate parabolic problems, we note that the Stampacchia-type iterative methods that we see in the stochastic case are inspired by a number of works on deterministic problems, mainly thin-film equations \cite{Shishkov_Hulshof_FSOP, GS, Grun2001b, DalPasso_Giacomelli_Grun_PISA, Passo_Giacomelli_Shishkov, AnsiniGiacomelli2004, Giacomelli_Grun_WTP_general}. We also point out the works \cite{GrillmeierDissertation,GK_fsop3}, which apply the stochastic energy method to other degenerate parabolic SPDEs. The approach of using differential inequalities to describe support propagation has also found success in the study of deterministic degenerate parabolic equations \cite{ChipotSideris1985, FischerOptRate, FischerUpperBoundWTP, Bernis2}. So far, this approach has found less success in the stochastic setting, with the result \cite{GrillmeierDissertation} being the only such example to the best of our knowledge. Stochastic filtering techniques, applied to other quantities than integral functionals, have also been employed in \cite{SHE1,SHE2,SHE3} to study support properties of stochastic heat equations with source type It\^o noise, which is sufficiently non-degenerate near $u=0$.

Since their discovery, stochastic flows of diffeomorphisms have shown themselves to be very useful in the study of SPDEs with stochastic transport terms such as \eqref{Eq_SPME_intro}; see \cite{Tubaro} for instance. Nowadays, they are frequently applied in regularity theory for SPDEs \cite{ASV_preprint,Pesce_fundamental} or in the study of their dynamical aspects \cite{gess2026ergodicityspdesdrivendivergencefree}. Interestingly, the particular idea to absorb effects of stochastic transportation into a test function has been also very successful in establishing a well-posedness theory for conservative SPDEs driven by rough paths \cite{LPS_1,LPS,gess_soug_sscl_1,FG_pathwise}, where it is used to give sense to the kinetic formulation of the equation.

\subsection{Notation}

Throughout this manuscript, 
$(\Omega, \mathfrak{A}, \P)$ is a fixed, complete probability space with a right-continuous and complete filtration $\mathscr{F}$. The family $(\beta_k)_{k\in \N}$ consists of independent, $\R^d$-valued $\mathscr{F}$-Brownian motions and we write $\E$ for the evaluation of the expectation with respect to $\P$. We remark that when denoting integration against the Lebesgue measure we omit  the differentials $\d x $ and/or $\d t$ to shorten our formulas whenever there is no risk of confusion.

We also record, the following recurring objects and where they are introduced:
\begin{itemize}
    \item We write $\T^d$ for the flat torus $\R^d/2L\Z^d$, where $L\in (0,\infty)$ is arbitrary but fixed throughout the manuscript. We use $\pi$ to denote the quotient map $\pi:\R^d\ra\T^d$.
    \item 
    The property of a measurable function $\vp\colon  \T^d\to \R$ to have {compact effective support} is defined in Definition \ref{defi_CES}.
    \item $T\in (0,\infty)$ denotes throughout a fixed, finite time horizon.
    \item Kinetic solutions to \eqref{Eq_SPME_intro}, denoted by $u$, are defined in Definitions \ref{defi_kin_measure}--\ref{defi_kin_sol}.
    \item (Parabolic) waiting times $T_w^{(\phi)}(B_r(x_0)) $ are for $x_0\in \T^d\setminus \supp (u_0)$ defined in Definition \ref{def:parabolic_wt}.
    \item The class of functions $ \mathscr C_{\delta,\kappa}$ and the set $ \mathscr S_{x_0}$ are for $\delta>0$, $\kappa>1$ and $x_0\in \T^d\setminus \supp (u_0)$ defined in \eqref{GG-137}
 and \eqref{eq:S}, respectively. 
 \item Stochastic flow of diffeomorphisms, denoted by $\phi_t$ as well as solutions to the stochastic transport equation, denoted by $\eta$, are introduced in Section \ref{S:flows}.
    \item The regularized noise coefficient $\sigma_l$ is chosen in \eqref{Eq_21} based on a suitable cutoff function $\zeta$.
    \item Weak solutions to \eqref{Eq_viscous_reg}, denoted by $u_{\kappa,l}$, are defined in Definition \ref{def_weak_sol}.
    \item Suitable initial data for the random localization function $\varphi_{s',\delta'}$ and $\varphi_{y,\varepsilon}$ are defined at the beginning of the proof of Proposition \ref{prop:stampacchia} and Section \ref{SS:proof_lower}, respectively.
\end{itemize}

We also remark that throughout, we do not emphasize dependence on the thereby fixed parameters $d$ and $L$ in implicit constants. 

\section{Rigorous statement of the main results}\label{SS:Assumptions_Main_Results}
Let us introduce our assumptions on the data of \eqref{Eq_SPME_intro}. We assume throughout that 
    \begin{equation}\label{eqn_ID}\tag{A1} u_0 \in L^{2}(\T^d), \quad u_0 \ge 0,
 \end{equation}
 as well as 
\begin{equation}\label{equation_m}\tag{A2}
	m\in (1,\infty)  . \end{equation}
From the noise in \eqref{Eq_SPME_intro} we demand that
    \begin{align}
     W(t,x) \, = \, \sum_{k=1}^\infty \psi_k(x) \beta_k(t) ,\qquad &\psi \in C^{2+\alpha}( \T^d ; \ell^2 ) \quad \text{and} \quad \sum_{k=1}^\infty\psi_k \nabla \psi_k \in C^{2+\alpha}( \T^d ; \R^d ),
    \tag{A3}\label{eqn_ass_noise}
\end{align}
for some $\alpha\in (0,1)$.

 The above implies in particular the condition   \cite[Assumption 2.1]{Fehrman_Gess_ARMA} on the noise coefficients there and we recall that  the It\^o-formulation of \eqref{Eq_SPME_intro} then reads 
\begin{equation}\label{Eq_sPME_Ito}
	\d u \,=\, \Delta (u^m) \,\d t 
	\,+\,\frac{1}{2} \diver\bigl(
	\Psi_1  \nabla u \,+\, u \Psi_2
	\bigr)\, \d t
	\,-\, \sum_{k=1}^\infty \diver(  u \psi_k  \, \d \beta_k),
\end{equation}
where we define the quantities  
\begin{align}
    \Psi_1 =\sum_{k=1}^\infty \psi_k^2,\qquad 
    \Psi_2 = \frac{1}{2}\sum_{k=1}^\infty \nabla (\psi_k^2),\qquad \Psi_3 = \sum_{k=1}^\infty |\nabla \psi_k|^2 ,
\end{align}
cf. \cite[Section 2]{Fehrman_Gess_ARMA}. 

To rigorously state our results, we recall the definitions of kinetic measures and kinetic solutions to \eqref{Eq_SPME_intro} from \cite[Definitions 3.1, 3.2]{Fehrman_Gess_ARMA}, which can be obtained by formally computing the evolution of the {kinetic function} $(t,x,v )\mapsto \mathbf{1}_{\{0<v <u(t,x)\}} $ by means of It\^o's formula.  There,  $\delta$ denotes the Dirac mass at $0$ and $\beta_k^{(i)}$ the $i$-th entry of the $\R^d$-valued Brownian motion $\beta_k$. 

\begin{defn}[Kinetic measure]\label{defi_kin_measure}
A kinetic measure is a mapping $q$ from  $\Omega$  to the space of non-negative, locally finite measures on  $\T^d \times (0,\infty)\times[0,T] $ such that
\begin{equation}
(\omega,t) \,\mapsto \int_0^t \int_{\R} \int_{\T^d}  \psi (x,v ) \,\d q(\omega)
\end{equation}
defines an $\mathscr{F}$-predictable process, for all $\psi \in C_c^\infty( \T^d\times(0,\infty))$.
\end{defn}
\begin{defn}[Kinetic solution to \eqref{Eq_SPME_intro}]\label{defi_kin_sol}Let $u_0\in L^1(\T^d)$ be non-negative and assume  \eqref{equation_m}  and \eqref{eqn_ass_noise}. Then, a kinetic solution to \eqref{Eq_SPME_intro} is a non-negative, continuous, $\mathscr{F}$-adapted, $L^1(\T^d)$-valued process $u\in L^1(\Omega \times [0,T]\times \T^d)$, such that:
	\begin{enumerate}[label=(\roman*)]
	\item mass is conserved, i.e., $\P$-a.s, we have  $\|u(t)\|_{L^1(\T^d)} = \|u_0\|_{L^1(\T^d)}$ for all $t\in [0,T]$,
	\item it holds $u \in L^2(\Omega \times [0,T]\times \T^d)$,
	\item  we have $((K\wedge u )\vee 1/K) \in  L^2(\Omega \times [0,T]; H^1(\T^d))$ for all $K\in \N$,
	\end{enumerate}  
and there exists  a kinetic measure $q$ with 
\begin{enumerate}[label=(\roman*)]
	\setcounter{enumi}{3}
	\item $\P$-a.s, $m\delta(v - u)v^{m-1} |\nabla u|^2 \le q $ in the sense of measures on $\T^d \times (0,\infty) \times [0,T]$,
	\item \label{Cond_smallness_q} $q$ vanishes at infinity in the sense that $\lim_{M\to\infty} \E [ q(\T^d \times [M,M+1] \times [0,T])] = 0$,
	\item and for every $\psi \in C_c^\infty(\T^d\times (0,\infty ))$, it holds $\P$-a.s.\ 
	\begin{align}\begin{split}\label{Eq_kin_formulaiton}&\int_{\R}\int_{\T^d}
		\mathbf{1}_{\{0<v <u(t,x)\}} \psi(x,v)
		\,\d x \d v \,=\, \int_{\R}\int_{\T^d}
		\mathbf{1}_{\{0<v <u_0(x)\}} \psi(x,v)
		\,\d x \d v\\&\qquad -\,m\int_0^t \int_{\T^d}
		u^{m-1}\nabla u \cdot (\nabla \psi )(x, u)
		\,
		\d x
		\d s\,-\, \frac{1}{2}\int_0^t \int_{\T^d}\bigl[\Psi_1  \nabla u \,+\, u\Psi_2 \bigr] \cdot (\nabla \psi )(x, u) \,\d x \d s
		\\&\qquad -\,\int_0^t  \int_{\R} \int_{\T^d} \partial_v \psi (x,v) \,\d q\,+\,\frac{1}{2}\int_0^t \int_{\T^d} \bigl(
		u\nabla u \cdot \Psi_2 \,+\, u^2\bigr) (\partial_v\psi)(x,u) 
		\,\d x
		\d s
		\\&\qquad - \sum_{i=1}^d\sum_{k=1}^\infty \int_0^t \int_{\T^d} \psi(x, u) \partial_i ( u\psi_k) \,\d x \d \beta_k^{(i)}		
		,
		\end{split}
	\end{align}
	for all $t\in [0,T]$.
\end{enumerate}
\end{defn}
The above assumptions \eqref{eqn_ID}--\eqref{eqn_ass_noise} guarantee  by \cite[Theorem 4.6]{Fehrman_Gess_ARMA} the existence of a unique {kinetic solutions} to \eqref{Eq_SPME_intro}. 

We now state our main results, for which we need to introduce the concept of  parabolic waiting times. The idea is to examine waiting times {after} modulating out the effects of the hyperbolic part of \eqref{Eq_SPME_intro}. We recall that the appearing constant $L$ is half the torus length.

\begin{defn}
    \label{def:parabolic_wt}
    Let $u$ be the unique kinetic solution to \eqref{Eq_SPME_intro}, $x_0 \in \T^d\setminus \supp(u_0)$, and define 
    \[
        R_0(x_0) := \dist{x_0}{\supp(u_0)}.
    \]
    Given $r\in (0,R_0(x_0)\wedge L]$ and the unique stochastic flow of diffeomorphisms $\phi$ solving \eqref{eqn_flowww}, we define the \emph{parabolic waiting time of $u$ associated with} $B_r(x_0)$ to be 
    \[
        T_w^\phi(B_r(x_0)) := \inf\left\{t>0,\quad \int_{\phi_t(B_r(x_0))} u(t,x) \, \d x > 0\right\} \, \wedge  \, T. 
    \]
    The (absolute) waiting time of $u$ associated with $B_r(x_0)$ is given by 
    \[
        T_w(B_r(x_0)) := \inf\left\{t>0,\quad \int_{B_r(x_0)} u(t,x) \, \d x > 0\right\} \, \wedge  \, T.
    \]
\end{defn}
For details on the existence, uniqueness  and properties of stochastic flows of diffeomorhpisms, see Section \ref{S:flows}. 
{
\begin{remark} To convince ourselves that the above definition has the desired effect, let us consider non-negative solutions  to the purely hyperbolic SPDE $$\d u = -\diver( u \circ \d W).$$ For $W$ that is sufficiently regular in space, the identity $\d \int \eta u = 0$, where $\eta $ is a solution to \eqref{eqn_tpe}, shows that $ \int_{\phi_t(B_r(x_0))} u = \int_{B_r(x_0)} u_0$ by taking $\eta_0 \uparrow \mathbf{1}_{B_r(x_0)}$ and the properties of $\eta$ reviewed in Subsection \ref{ss:method}. Thus, the parabolic waiting time takes only trivial values,  $T^\phi_w(B_r(x_0)) \in\{ 0,\infty\}$ depending on whether $ \int_{B_r(x_0)} u_0  =0$.

\end{remark}
}

In order to state our main results, we need two technical definitions. First, for parameters $\delta>0$ and a number $\kappa>1$, define the set
\begin{equation}\label{GG-137}
        \mathscr C_{\delta,\kappa}:= \left\{ f:[0,\delta] \ra \ \R_+ \, \bigg\vert \, \exists(a_k)_{k\in \N} \subset \R_+ \, \text{ s.t. } \,  \,\sum_{k\in \N} a_{k-1}^{\kappa}a_k\inv  < +\infty, \, \sum_{k\in \N} f\left(\delta \cdot 2^{-k}\right)a_k^{-1} < +\infty \right\}.
\end{equation}
The above captures functions $f$ with a mild, but sufficiently fast decay at $0$. In fact, the above class contains always functions that decay slower to $0$ than any power law, see  \cite[Remark 2.8]{GSU_SPME_Superlinear} for an explicit example.

Given $x_0\in\T^d\setminus \supp(u_0)$ such that $R_0(x_0)<L$, we also define the set $\mathscr S_{x_0}$ by
\begin{equation}
    \label{eq:S}
    \mathscr S_{x_0} \, = \, \left\{ (y,\epsilon) \in B_{R_0}(x_0) \times [0,L], \quad
             \epsilon^2>\dist{y}{\partial B_{R_0}(x_0)}^2 \geq \frac{d+1}{d+2}\epsilon^2\right\}.
\end{equation}
Note that $\mathscr S_{x_0} \neq \emptyset$, because we always have $\left(x_0, R_0+\delta\right) \in \mathscr S_{x_0}$ for a small enough $\delta>0$.

We start by stating our results on finite speed of propagation of \eqref{Eq_SPME_intro} and sufficient conditions for parabolic waiting time phenomena.
\begin{thm}\label{thm:lower}
    If assumptions \eqref{eqn_ID}--\eqref{eqn_ass_noise} are satisfied and  $x_0\in \T^d\setminus\supp(u_0)$, then the following holds: 
    \begin{enumerate}[label = (\roman*)]
        \item\label{thm:point_1} For any $r\in (0,R_0(x_0)\wedge L)$, the parabolic waiting time associated with $B_r(x_0)$ is almost-surely positive. 
        \item\label{thm:point_2} $u$ has finite speed of propagation, i.e., for all $r\in (0,R_0\wedge L)$, the waiting time $T_w(B_r(x_0))$ is positive with probability one. 
        \item\label{thm:point_3} If in addition $R_0<L$ and there exist $r_0>0$ and $f\in \mathscr C_{2r_0,\frac{m+1}{2}}$ such that 
        \begin{equation}\label{eqn_11}
            \sup_{r\in (0,r_0)} \, \frac{1}{f(r)\, r^{d+\frac{4}{m-1}}} \int_{B_{R_0+r}(x_0)\setminus B_{R_0}(x_0)} u_0^2 \,  < +\infty,
        \end{equation}
        then $T_w^\phi(B_{R_0}(x_0)) > 0$ with probability one. 
    \end{enumerate}
\end{thm}

The condition \eqref{eqn_11} asks that the local mean of $u_0^2$ decays like $r^{4/(m-1)}$ near the boundary of $\supp(u_0)$. One might then suspect that a critical power-law-type growth for $u_0$ should have exponent $2/(m-1)$. Under a stronger condition on the initial data, this is true, as is displayed in the following corollary.

\begin{col}\label{col:lower}
    Let assumptions \eqref{eqn_ID}--\eqref{eqn_ass_noise} be satisfied. If $x_0\in \T^d\setminus\supp(u_0)$ such that $R_0<L$ and there exist $r_0>0$, $f\in \mathscr C_{2r_0,\frac{m+1}{2}}$ such that 
        \[
            \sup_{r\in (0,r_0)} \, \frac{1}{\sqrt{f(r)}\, r^{\frac{2}{m-1}}} \sup_{B_{R_0+r}(x_0)\setminus B_{R_0}(x_0)} u_0(x) \,  < +\infty,
        \]
        then $T_w^\phi(B_{R_0}(x_0)) > 0$ with probability one. 
\end{col}

Next,  we give a sufficient condition for the absence of a parabolic waiting time phenomenon (i.e., instant forward motion even after modulating out the stochastic flow). 

\begin{thm}\label{thm:upper}
    Let assumptions \eqref{eqn_ID}--\eqref{eqn_ass_noise} be satisfied and $x_0\in \T^d\setminus \supp(u_0)$ such that $R_0<L$. If 
    \begin{equation}\label{eqn_12}
        \inf_{\mathscr S_{x_0}} \,\left\{ \left(\frac{1}{\epsilon^{4+d+\frac{2}{m-1}}} \tint \left( \epsilon^2 -\dist{x}{y}^2\right)_+^2\, u_0(x) \, \d x\right)^{1-m} \right\} \,  = \, 0
    \end{equation}
    then $T_w^\phi(B_{R_0}(x_0)) = 0$ with probability one.
\end{thm}

    Heuristically speaking, \eqref{eqn_12} asks that the mean value of $u_0$ grows strictly faster than that of a power law profile with exponent $2/(m-1)$ near the boundary of $\supp(u_0)$. We capture also this by a stricter, but hopefully more intuitive assumption on $u_0$.

\begin{col}\label{col:upper}
    Let assumptions \eqref{eqn_ID}--\eqref{eqn_ass_noise}. Suppose that $x_0\in \T^d\setminus \supp(u_0)$ so that $R_0<L$ and $\tilde x \in \partial B_{R_0}(x_0)\cap \supp(u_0)$. If there exists a truncated cone $\mathcal C(\tilde x)$ with vertex $\tilde x$ such that 
    \begin{align}\label{eqn_bshfdsf}
        \limsup_{\epsilon\downarrow 0} \, \inf_{ B_\eps(\tilde x)\cap \mathcal C(\tilde x)}\, \frac{u_0(x)}{\dist{x}{\tilde x}^{\frac{2}{m-1}}} = +\infty,
    \end{align}
    then $T_w^\phi(B_{R_0}(x_0)) = 0$ with probability one.
\end{col}
We conclude this section, by illustrating the announced near-dichotomy for power law type initial data given by the conditions of Corollaries \ref{col:lower} and \ref{col:upper} in the one-dimensional case.  Let us identify $\T^1 = [-L,L]$, and consider initial data $u_0$ which takes the form
    \[
        u_0(x) = \begin{cases}
            x^\gamma, & x>0 ,\\
            0, & x\leq 0,
        \end{cases}
    \]
    in a neighborhood of zero. Setting $x_0 = -\epsilon $ for some small $\epsilon>0$, the condition of Corollary \ref{col:lower} is satisfied if and only if an interval $(0,r_0)$ exists in which
    \[
    \sup_{(0,r_0)}\, \frac{r^\gamma}{\sqrt{f(r)}r^{\frac{2}{m-1}}}  
    \]
    is bounded. Since we know that there exist $f$ such that $f\inv$ blows up slower than any inverse power law, the least-stringent condition on $\gamma$ such that the above is bounded would be
    \[
        \gamma > \frac{2}{m-1},
    \]
    which ensures an almost-surely positive waiting time for the interval $(-2\epsilon,0)$. On the other hand, the condition of Corollary \ref{col:upper} (with $\tilde x=0$) tells us that if 
    \[
        \gamma < \frac{2}{m-1},
    \]
    then the waiting time of $(-2\epsilon,0)$ is almost-surely zero. 

\section{Preliminaries on stochastic  flows of diffeomorphisms and transport equations}
\label{S:flows}
In this section, we consider the solution to the linear stochastic transport equation \eqref{eqn_tpe}, 
    which may in light of \eqref{eqn_ass_noise} be expressed in the It\^o-form 
    \begin{align} \label{eqn_transport}
    \d \eta 
    = \frac{1}{2} \bigl(
	\diver (\Psi_1  \nabla \eta  ) -{\Psi_2} \cdot \nabla \eta 
	\bigr) \, \d t
	- \sum_{k=1}^\infty \psi_k  \nabla \eta \cdot \d \beta_k.
\end{align}
We additionally examine the closely related stochastic  flow of diffeomorphisms induced by the SDE
\begin{align} \label{eqn_ODE_strat}
        \d \phi(x)  =  \sum_{k\in \N} \psi_k(\phi(x)  ) \circ \d \beta_k, \qquad \phi_0(x) = x,
\end{align}
 the It\^o-form of which is
\begin{align}\label{eqn_ODE_ito}
        \d \phi(x) =
    \frac{1}{2}\Psi_2  (\phi(x) ) \, \d t +
\sum_{k\in \N} \psi_k(\phi (x)) \,\d \beta_k, \qquad 
\phi_0(x) =  x.
\end{align}
\begin{remark}\label{rem_interpretation_of_per_flow}
We remark that while \eqref{eqn_ODE_strat} may be given a sense as a $\T^d$-valued SDE, we will in the following understand it as the SDE \eqref{eqn_ODE_ito} posed on $\R^d$ with the coefficient functions replaced by their periodic extensions on $\R^d$ subject to the constraint that $\phi_t$ respects $\T^d$-equivalence classes. 
The latter is natural, since the periodicity of the coefficient functions $\Psi_2$ and  $(\psi_k)_{k\in\N}$ implies that $\phi_t(x)+j$ solves the some equation if $\phi_t(x)$ satisfies \eqref{eqn_ODE_ito}, for $j\in 2L\Z^d$. The resulting  \eqref{eq:periodic_phi} allows us in turn to identify such a  $\phi_t$ again with a mapping  $\phi_t\colon \T^d \to\T^d$.
\end{remark}

Stochastic flows of diffeomorphisms, and their relation to transport equations,
were originally explored by Kunita, which lead to his  seminal work \cite{kunita_book}. Since the latter uses an impressive but  non-standard notion of stochastic calculus, we rely for the sake of accessibility rather on the precursor \cite{Kunita} and the versions thereof presented in \cite{ASV_preprint}. 

\begin{lemma}\label{lemma_flows}
    There exist modifications of the solutions $\phi_t(x)$ to \eqref{eqn_ODE_ito} posed on $\R^d$ such that $\P$-a.s.\ the following is satisfied:
	\begin{enumerate}[label=(\roman*)]
    \item\label{Item1} For each $t\in [0,T]$
\begin{equation}\label{eq:periodic_phi}
		\phi_t(x+j) \,=\, \phi_t(x) + j, \qquad j\in {2L}\Z^d,\,x\in \R^d,
\end{equation}
in particular, $\phi_t$ respects $\T^d$-equivalence classes.
		\item\label{Item2} The mapping
		\[
		\phi\colon 
		[0,T]\times \R^d \to \R^d ,\quad (t,x) \mapsto \phi_{t}(x)
		\]
		lies in $C_{\loc}^{1/2-, (2+\alpha)-}([0,T]\times \R^d;\R^d)$. 
		\item\label{Item3} For each  $t \in [0,T]$, the map
		\begin{equation*}\label{Eq:homeomorph}
		\phi_{t} \colon \R^d \to \R^d
		\end{equation*}
		is a $C^2$-diffeomorphism. 
		\item\label{Item4}  The inverses $ \phi^{-1}$ 
	lie in $C_{\loc}^{1/2-,(2+\alpha) - }([0,T]\times \R^d ;\R^d)$.
	\end{enumerate}
\end{lemma}
\begin{proof}
The preiodic extensions of $\Psi_2$ and $(\psi_k)_{k \in \N}$ satisfy \cite[Assumption 2.1]{ASV_preprint} and \cite[Assumption 2.5]{ASV_preprint} with $k=2$ and $\alpha$ from \eqref{eqn_ass_noise}. Thereby, the regularity assertion \ref{Item2} follows from \cite[Theorem 2.3 (i) \& Theorem 2.6 (i)]{ASV_preprint} while \ref{Item3} is the content of \cite[Theorem 2.6 (ii)]{ASV_preprint}. The periodicity property \ref{Item1} follows then by uniqueness of such a process $\phi$ together with the observation of Remark \ref{rem_interpretation_of_per_flow} that also its shifted version solves \eqref{eqn_ODE_ito} with shifted initial data. Regarding the last property, we note that \cite[Proposition 2.8]{ASV_preprint} yields $C^{1/2-,1}_\loc([0,T]\times \R^d ;\R^d)$-regularity of $\phi^{-1}$. The inverse function rule $D\phi_t^{-1}(y)= (D\phi_t)^{-1}( \phi_t^{-1}(y))$ together with \ref{Item2} implies then that $D\phi_t^{-1}$ lies in $C^{1/2-,1}_\loc([0,T]\times \R^d ;\R^{d\times d})$. Its second order version 
\begin{align}
    \partial_{ij} \phi_t^{-1} (y) = \bigl(D\phi_t^{-1} \partial_{kl} \phi_t \bigr)(\phi_t^{-1}(y)) (D \phi_t^{-1})_{k,i}(y) (D \phi_t^{-1})_{l,j}(y),
\end{align}
gives also $\partial_{ij} \phi^{-1} \in C^{\alpha/2-,\alpha-}_\loc([0,T]\times \R^d ;\R^{d\times d})$ and thereby the claimed \ref{Item4}.
\end{proof}

We turn our attention to the stochastic transport equation \eqref{eqn_tpe}. While we apply it later to more regular data, it is most naturally formulated in $H^1(\T^d)$.

\begin{defn}\label{def:solution_te}An $\mathscr{F}$-adapted, weakly continuous $H^1(\T^d)$-valued process process is called a weak solution to \eqref{eqn_tpe}, if for any $\varphi\in C^\infty(\T^d)$ we have $\P$-a.s.\ 
\begin{align}
    \int_{\T^d} \eta(t) \varphi \, \d x
    = 
    \int_{\T^d} \eta_0 \varphi \, \d x - \frac{1}{2}\int_0^t \int_{\T^d}\Psi_1 \nabla \eta \cdot \nabla \varphi + \varphi \Psi_2 \cdot \nabla \eta \,\d x  
    \,\d s -\sum_{k=1}^\infty \int_0^t \int_{\T^d}  \varphi \psi_k \nabla \eta \,  \d x \cdot \d \beta_k ,
\end{align}
for all $t\in [0,T]$.
\end{defn}
\begin{lemma}\label{lemma:representation_formula}  For each $\eta_0 \in H^{1}(\T^d) $,  the equation \eqref{eqn_tpe} has a unique weak solution in the sense of \eqref{def:solution_te}. Moreover, it admits $\P$-a.s.\ the representation \begin{equation} \label{eqn_repr_by_stoch_flow}  \eta(t,x) = \eta_0 (\phi_t^{-1}(x)) ,  \qquad (t,x)\in [0,T]\times\T^d.\end{equation}
\end{lemma}
 \begin{proof}The well-posedness of \eqref{eqn_tpe} is well-known: Due to the noise regularity \eqref{eqn_ass_noise}, uniqueness may be established by an $L^2(\T^d)$-energy estimate for the difference of two solutions, while existence can be shown using a viscous regularization procedure. For a rigorous statement we refer the reader for instance to \cite[Theorem 2.1]{Ger_Gyong_kryl}  stating classical results %by Krylov and Rozovskii 
 from \cite{krylov1986characteristics}.  Indeed, we may apply this result with $m=1$ and $p=q=2$ since the required \cite[Assumptions 1--3]{Ger_Gyong_kryl} follow from \eqref{eqn_ass_noise}. Then, to derive the representation \eqref{eqn_repr_by_stoch_flow}, we may consider the periodic extension of $\eta$ to $\R^d$ and compose it with $\phi$ by means of the It\^o--Wentzell formula from \cite[Proposition 2.14]{ASV_preprint}, which yields that $\P$-a.s.
 
 \begin{align}\begin{split}\label{eqn_IW}
     \eta(t) \circ \phi_t -\eta_0 =& \int_0^t \frac{1}{2} \bigl(
	\diver (\Psi_1  \nabla \eta  ) -{\Psi_2} \cdot \nabla \eta 
	\bigr) 
    \circ \phi\, \d s  
	-\sum_{k=1}^\infty \int_0^t (\psi_k  \nabla \eta) \circ \phi  \cdot \d \beta_k \\&
    + \frac{1}{2}\int_0^t (\Psi_2\cdot \nabla \eta )\circ \phi \, \d s + \sum_{k=1}^\infty  \int_0^t(\psi_k \nabla \eta) \circ \phi \cdot\d \beta_k 
    \\& - \sum_{k=1}^\infty \int_0^t 
    \bigl( \psi_k \diver( \psi_k  \nabla \eta ) \bigr)\circ \phi
    \,\d s+ \frac{1}{2}\int_0^t (\Psi_1 \Delta \eta)\circ\phi \,\d s,\end{split}
 \end{align}
 as continuous $H^{-1}_\loc(\R^d)$-valued processes defined for $t\in [0,T]$. We also remark that in the first, fifth and sixth term on the right-hand side, the composition with $\phi$ needs to be interpreted distributionally in the sense that
 \[
\langle f\circ \phi ,\varphi  \rangle :=\langle f, \varphi\circ \phi^{-1} |\det D \phi^{-1}|\rangle,
\]
which preserves $H^{-1}_\loc(\R^d)$, see the comments above \cite[Proposition 2.13]{ASV_preprint}. In any case, all the terms on the right-hand side of \eqref{eqn_IW} cancel, since 
\begin{align}
    \frac{1}{2}\Psi_1 \Delta \eta - \sum_{k=1}^\infty \psi_k \diver( \psi_k  \nabla  \eta ) = \frac{1}{2}\diver(\Psi_1 \nabla \eta)  - \Psi_2 \cdot \nabla \eta  - \diver(\Psi_1 \nabla \eta ) + \Psi_2 \cdot \nabla \eta = -\frac{1}{2} \diver(\Psi_1 \nabla \eta).
\end{align}
This implies that $\eta(t)\circ \phi_t = \eta_0$ as periodic functions on $\R^d$, and thereby also on $\T^d$.
 \end{proof}

\section{(In)equalities for the randomly-localized mass and energy}\label{SS:prop_kinetic}

A main ingredient in the proofs of the main results are randomly localized mass identities and energy estimates for solutions to \eqref{Eq_SPME_intro}. By this we mean localized integral (in)equalities involving $u$, where the localizing function is randomly determined as a solution to \eqref{eqn_tpe}. To derive them, we utilize the It\^o rule from Appendix \ref{app:ito} together with the viscous regularization of \eqref{Eq_SPME_intro} used in \cite{Fehrman_Gess_ARMA} as approximations to construct solutions to \eqref{Eq_SPME_intro}. The latter is given by    
\begin{equation}\label{Eq_viscous_reg}\tag{sPME-reg}\begin{cases}
  \d u \,=\, \Delta (u^m) \,\d t \,+\,\kappa  \Delta u\,\d t 
    \,+\,\frac{1}{2} \diver\bigl(
    \Psi_1 (\sigma_l'(u))^2 \nabla u \,+\, \sigma_l'(u)\sigma_l(u) \Psi_2
    \bigr)\, \d t
    \,-\, \diver\left( \sigma_l(u)\, \d W\right),\\
    u(0) \,=\, u_0,
    \end{cases}
\end{equation}
where $\kappa >0$ and $l\in \N$. The coefficients $\sigma_l$ are appropriate regularizations such that $\sigma_l(r)\ra r$ and $\sigma_l'(r) \ra 1$ uniformly on compact subsets of $(0,\infty)$. Furthermore, they are assumed to  satisfy \cite[Assumption 5.2]{Fehrman_Gess_ARMA} uniformly in $l$ as well as 
\begin{equation}\label{Eq_ass_sigma_l}
    \sigma_l \in C([0,\infty)) \cap C^\infty((0,\infty)) ,\qquad \sigma_l(0) \, =\,  0,\qquad \sigma_l' \in C_c^\infty([0,\infty)).
\end{equation}
For the purpose of this manuscript, we make the explicit choice  
    \begin{equation}\label{Eq_21}
    \sigma_l(r)\,=\, r \zeta(r/l),
    \qquad
    r\in [0,\infty),
    \end{equation}
     for a fixed function $\zeta\in C_c^\infty([0,\infty))$ that is decreasing with $\zeta\equiv 1$ on $[0,1]$ and $\zeta \equiv 0$ on $[2 ,\infty)$. One computes 
\begin{align}\label{eqn_21*}
\sigma_l'(r) \,=\, \zeta(r/l) \,+\, r \zeta'(r/l)/l,
\end{align}
and so 
\begin{align}\begin{split}\label{Eq_bounds_sigma_l}
0\,\le \, \sigma_l(r)\,&\le \,r,
\qquad 
|\sigma_l'(r)|
\,\lesssim_{\zeta} \,1,
\end{split}
\end{align}
 for $r\ge 0$, from which we infer that this choice is admissible.  Weak solutions to \eqref{Eq_viscous_reg} are then defined as follows, cf.\ \cite[Defintion 5.6]{Fehrman_Gess_ARMA}. 
\begin{defn}[Weak solutions to \eqref{Eq_viscous_reg}]\label{def_weak_sol}
    A weak solution to \eqref{Eq_viscous_reg} is a continuous, $L^{2}(\T^d)$-valued, adapted, and non-negative process $u$ satisfying $u, u^{(m+1)/2}\in L^2(0,T; H^1(\T^d))$ such that for any $\varphi\in C^\infty(\T^d)$, $\P$-a.s., for every $t\in [0,T]$:
    \begin{align*}
        \int_{\T^d} u(t) \vp \,\d x \,=\,& \int_{\T^d} u_0 \vp \,\d x \,-\, m \int_0^t \int_{\T^d} u^{m-1} \nabla u \cdot \nabla \vp \,\d x\d s
        \,-\,\kappa \int_0^t \int_{\T^d} 
        \nabla u \cdot  \nabla \vp  \,\d x\, \d s
        \\& -\,\frac{1}{2}\int_0^t \int_{\T^d}
        \Psi_1(\sigma_l'(u))^2 \nabla u \cdot \nabla \vp \,+\,
        \sigma_l(u)\sigma_l'(u)\Psi_2 \cdot\nabla\vp
        \, \d x\,\d s\\&+\, \sum_{k=1}^\infty\int_0^t \int_{\T^d}
        \sigma_l(u)\psi_k \nabla \vp \, \d x \cdot  \,
        \d \beta_k.
    \end{align*}
\end{defn}
Due to the additional Laplacian and the more regular noise term,  weak solutions $u_{\kappa,l}$ to \eqref{Eq_viscous_reg} can be constructed using a Galerkin scheme under the above assumptions as shown in \cite[Proposition 5.17]{Fehrman_Gess_ARMA}. 
In \cite[Section 5.3]{Fehrman_Gess_ARMA}, it is moreover proved that these weak solutions converge to the unique kinetic solution $u$ to \eqref{Eq_SPME_intro} as  $\kappa\to 0$ and $l\to\infty$. 
More precisely, inspecting the proof of \cite[Theorem 5.25]{Fehrman_Gess_ARMA} shows that we have
\begin{align}\begin{split}\label{Eq_convergences}
    u_{\kappa,l} \,\to \, u ,\qquad &\text{in }L^1[0,T ; L^1(\T^d)), \\
    \nabla (    u_{\kappa,l}^{{(m+1)}/2} ) \,\to\,
    \nabla( u^{{(m+1)}/2}  )
    ,\qquad &\text{in }(L^2([0,T]\times\T^d; \R^d) , d_{w.b}),% \\
\end{split}\end{align}
in probability for any sequence $\kappa\to 0$ and $l\to \infty$, where the appearing metric $d_{w,b}$ is the metric given by \begin{equation}\label{Eq_d_wb}
d_{w,b}(u,v) \,=\, \sum_{k=1}^\infty 2^{-k} |\langle u - v,w_k\rangle |,
\end{equation}
for $(w_k)_{k\in \N}$ lying dense in the unit ball of $L^2([0,T]\times\T^d; \R^d) $. We remark that the latter metrizes weak convergence on bounded subsets of $L^2([0,T]\times\T^d) $, i.e., the convergence $u_n\rightharpoonup u$ is equivalent to boundedness of $\|u_n\|_{L^2([0,T]\times\T^d) }$ and convergence with respect to $d_{w.b}$.

An additional  convergence useful for our purposes has been derived in \cite{GSU_SPME_Superlinear}, namely that 
\begin{equation}
    \label{eq:GSU_convergences}
    u_{\kappa,l} \ra u ,\qquad \text{in }L^p( \Omega ; L^r([0,T] \times \T^d)), \qquad p<\infty, \quad  r< \left( 1+ \frac{2}{d}\right)(m+1).
\end{equation}
We  note that this is explicitly stated in \cite[Corollary 4.3]{GSU_SPME_Superlinear} for nonlinear noise coefficients $=u^n$ with $n>1$, but the exact same argument applies to  our case $n=1$.
{
We also recall the uniform estimate 
\begin{align}\label{eqn_boundedness_LinftyL2_subst}
    \sup_{\kappa \in (0,\infty) , l\in \N }\e{ \sup_{ t \in [0,T] } \| u_{\kappa,l}\|_{L^2(\T^d}^p} \,<\,\infty,\qquad p<\infty,
\end{align}
from \cite[eqn (4.12)]{GSU_SPME_Superlinear}. Together with \eqref{eq:GSU_convergences} and Lemma \ref{lemma_app_A},
}
this entails that in probability also 
\begin{align}\label{eqn_convergence_by_boundedness_LinftyL2}
 u_{\kappa,l} \to u \quad \text{in }L^{1/\delta} ( 0,T; L^{2- \delta}(\T^d)) ,\qquad \delta \in (0,1).
\end{align}

It is the higher regularity of the solutions to \eqref{Eq_viscous_reg} that allow us to justify an application of  It\^o's formula on this approximate level, cf.\ Proposition \ref{prop:Ito_product_rule}. Together with the aforementioned convergences, we obtain thereby the   randomly localized (in)equalities from Lemmas \ref{lemma:loc_mass}--\ref{lemma:loc_energy_general} for \eqref{Eq_SPME_intro}. To this end, for the remainder of this section, we let $\eta$ be the unique weak solution to \eqref{eqn_tpe} obtained from Lemma \ref{lemma:representation_formula} with initial data \begin{align} \label{initial_data_eta}
\eta_0\in W^{2,\infty}(\T^d).
\end{align} 
Moreover, we will always write $u$ for the unique kinetic solution to \eqref{Eq_SPME_intro} and $u_{\kappa,l}$ the solutions to \eqref{Eq_viscous_reg} and we recall that \eqref{eqn_ID}--\eqref{eqn_ass_noise} are assumed throughout.
\begin{lemma}
    \label{lemma:loc_mass}
    It holds $\P$-a.s.
    \begin{equation}
        \label{eq:loc_mass_general}
         \tint \eta(t) u(t) = \tint \eta_0u_0 +  \int_0^t\tint  u^m\Delta\eta ,
    \end{equation}
    for all $t\in [0,T]$.
 \end{lemma}

\begin{proof} 
    We start with the momentary additional assumption that 
    \begin{align}\label{eqn_add_q}\tag{qual}
    	u_0 \in L^{m+1}(\T^d),
    \end{align}
	which we remove at the end of this proof using the $L^1$-contraction estimate for kinetic solutions to \eqref{Eq_SPME_intro} established in \cite[Theorem 4.6]{Fehrman_Gess_ARMA}.  In order to compute $\d \int \eta \ukl $, we apply Proposition \ref{prop:Ito_product_rule} with $f(r) = r$, for which we set 
    \begin{align}
        &
        F = 0
        ,\\&
        G = \nabla \ukl^m + \kappa \nabla \ukl + \frac{1}{2} \Psi_1 (\sigma_l'(\ukl ))^2\nabla \ukl+ \frac{1}{2}\sigma_l'({\ukl}) \sigma_l (\ukl) \Psi_2,
        \\&
        H_k = - \nabla ( \psi_k  \sigma_l (\ukl)) ,
        \\&
        A = -\frac{1}{2}\Psi_2 \cdot \nabla \eta
        ,\qquad
        B = \frac{1}{2}\Psi_1 \nabla \eta
        ,\qquad
        E_k = -\psi_k \nabla \eta.
    \end{align}
	Regarding the required regularities
	\eqref{eqn_67} and  \eqref{eq:ass_ito_prod_regularity}, we observe that 
	\[
	\ukl \in C([0,T] ; L^2(\T^d) )\cap L^2(0,T;H^1(\T^d)) ,\qquad \ukl^{\frac{m+1}{2}} \in L^2(0,T;H^1(\T^d)),
	\]
	$\P$-a.s., 
	by Definition \ref{def_weak_sol}, and that the additionally imposed \eqref{eqn_add_q} guarantees  
	\begin{align}\label{tum}
	\ukl^m \in L^2(0,T;H^1(\T^d)), 
	\end{align}
	$\P$-a.s., 
	by \cite[Proposition 5.7]{Fehrman_Gess_ARMA}. Together with \eqref{eqn_ass_noise} and \eqref{Eq_bounds_sigma_l}, this implies that $\ukl$, $F$, $G$ and $H$  satisfy their respective assumptions for Proposition \ref{prop:Ito_product_rule} to hold. To conclude the same for the terms involving $\eta$, we make use of its representation formula \eqref{eqn_repr_by_stoch_flow} in terms of the stochastic flow of diffeomorphisms induced by \eqref{eqn_ODE_ito}. We have, using Einstein's summation convention, 
	\begin{align}\begin{split} &
		\eta(t) = \eta_0(\phi_t^{-1}) ,\qquad \partial_i \eta(t) = (\partial_i \phi_t^{-1})^j \partial_j \eta_0(\phi_t^{-1}) ,\\& 
		\partial_{ij} \eta(t) = (\partial_{ij} \phi_t^{-1})^k \partial_k \eta_0(\phi_t^{-1}) +(\partial_{i} \phi_t^{-1})^k (\partial_{j} \phi_t^{-1})^l  \partial_{kl} \eta_0(\phi_t^{-1}) .\\ \label{eqn_formula}
		\end{split}
	\end{align}
    The regularity of $\phi^{-1}$ proved in Lemma \ref{lemma_flows}~\ref{Item4} and the assumed \eqref{initial_data_eta} imply that $\P$-a.s.\
	\begin{equation}\label{regularity_of_eta}
		\eta \in C([0,T];C^1(\T^d) ) \cap L^\infty(0,T ; W^{2,\infty}(\T^d)).
	\end{equation}
	This, together with \eqref{eqn_ass_noise} yields that also $\eta$, $A$, $B$ and $E$ have the required regularity, and Proposition \ref{prop:Ito_product_rule} gives

    \begin{align}
        \tint \eta(t) \ukl(t) & = \tint \eta_0 u_0 
        - \int_0^t \tint \nabla\eta \cdot \left( \nabla \ukl^{m}+\kappa \nabla\ukl +  
        \frac{1}{2}\Psi_1(\sigma_l'(\ukl))^2 \nabla \ukl  + \frac{1}{2} \sigma_l'(\ukl )\sigma_l(\ukl)\Psi_2
        \right) \\
        &\quad - \frac{1}{2}\int_0^t \tint \ukl \Psi_2\cdot \nabla\eta  - \frac{1}{2}\int_0^t \tint \Psi_1\nabla\ukl \cdot \nabla\eta  + \sum_{k=1}^{\infty}\int_0^t \tint \psi_k\nabla\eta\cdot \nabla(\psi_k \sigma_l(\ukl)) \\
        &\quad - \sum_{k=1}^{\infty} \int_0^t \tint \psi_k\ukl \nabla\eta \cdot \d \beta_k + \sum_{k=1}^{\infty} \int_0^t \tint \psi_k \sigma_l(\ukl)\nabla\eta \cdot \d \beta_k . \label{eqn_abc}
    \end{align}
Introducing the notation 
\begin{align}\begin{split}  \label{eqn_Rdet}
	\mathscr R^{\mathrm{det}}_{\kappa,l}(t) & = -\kappa\int_0^t \tint \nabla\eta \cdot \nabla\ukl + \int_0^t \tint \Psi_1 \left( \sigma_l'(\ukl) - \frac{1}{2} (\sigma_l'(\ukl))^2 -\frac{1}{2} \right) \nabla \eta \cdot \nabla \ukl \\
	& \hspace{2cm} + \int_0^t \tint \left( \sigma_l(\ukl) - \frac{1}{2}\ukl - \frac{1}{2}\sigma_l'(\ukl) \sigma_l(\ukl)\right) \Psi_2 \cdot \nabla \eta,\end{split} \\
	\mathscr R^{\mathrm{stoch}}_{\kappa,l} (t)& = \sum_{k=1}^{\infty} \int_0^t \tint \psi_k \left(\sigma_l(\ukl)-\ukl\right)\nabla\eta \cdot \d \beta_k,\label{eqn_Rsto}
\end{align}
we may rewrite the above as
\begin{align}\label{eqn_71}
	\tint \eta(t)  u_{\kappa,l} (t) = 
	 \tint \eta_0 u_0 + \int_0^t \tint   u_{\kappa,l}^m\Delta\eta  +  \mathscr R_{\kappa,l}^{\mathrm{det}}(t) + \mathscr R_{\kappa,l}^{\mathrm{stoch}}(t).
    \end{align}

    We now wish to pass to the limit in the above terms. For this, we address them  separately, before removing the additionally imposed \eqref{eqn_add_q}.
	
	\textit{The surviving terms.} Since the first term on the right-hand side of \eqref{eqn_71} is also present in the desired \eqref{eq:loc_mass_general}, no argument needs to be made. Regarding the second term, we may leverage \eqref{eq:GSU_convergences} and a diagonal sequence argument to deduce that 
	\begin{align}\label{eqn_22*}
		{u}_{\kappa,l} \to u ,\qquad \text{in }L^r([0,T]\times \T^d), \qquad r<(1+2/d)(m+1),
	\end{align}
	 $\P$-a.s.,
	up to passing to a subsequence. Thereby, and by \eqref{regularity_of_eta}, we have $\P$-a.s., for all $t\in [0,T]$
	\begin{align}\int_0^t  u_{\kappa,l}^m \Delta\eta    \to 
		\int_0^t  u^m\Delta\eta    ,
	\end{align}
	i.e., the second term of the desired \eqref{eq:loc_mass_general}.
	
	\textit{The deterministic remainder.} Firstly, we notice that integration by parts, together with the convergences \eqref{eqn_22*}  yields that, $\P$-a.s., for all $t\in [0,T]$,
	\[
	\kappa\int_0^t \tint \nabla\eta \cdot \nabla u_{\kappa,l }  = - 
	\kappa\int_0^t \tint u_{\kappa,l }\Delta \eta   
	\to 0, 
	\] 
	due to the regularity of $\eta$ stated in \eqref{regularity_of_eta}. 

     Integration by parts in the second integral on the right-hand side of \eqref{eqn_Rdet} lets us express it as
    \[
        \int_0^t \tint \Psi_1 \left( \sigma_l'( u_{\kappa,l }) - \frac{1}{2} (\sigma_l'(u_{\kappa, l}))^2 -\frac{1}{2} \right) \nabla \eta \cdot \nabla  u_{\kappa,l} = - \int_0^t\tint \diver\left( \Psi_1\nabla\eta\right) \, \theta_l(u_{\kappa,l}),
    \]
    where
    \[
        \theta_l(r) := \int_0^r  \sigma_l'(s) - \frac{1}{2} (\sigma_l'(s))^2 -\frac{1}{2} \, \d s.
    \]
    Note that $\theta_l\ra 0$ as $l\ra \infty$ uniformly on bounded sets and satisfies $\theta_l(r) \lesssim r$, as follows from \eqref{Eq_21} and \eqref{eqn_21*}.
    Together with \eqref{eqn_ass_noise}, \eqref{regularity_of_eta}, and the convergence \eqref{eqn_22*}, this yields that, $\P$-a.s., for all $t\in [0,T]$,
    \begin{align}\label{eqn_kek}
        \int_0^t\tint \diver\left( \Psi_1\nabla\eta\right) \, \theta_l(u_{\kappa, l }) \ra 0.
    \end{align}

An analogous argument shows that also the last integral on the right-hand side of \eqref{eqn_Rdet} tends $\P$-a.s., for any $t\in [0,T]$, to $0$, so that the same can be said about $\mathscr R_{\kappa,l}^{\mathrm{det}}(t)$. 

	\textit{The stochastic remainder.}  We obesrve that the right-hand side of \eqref{eqn_Rsto} tends uniformly on $[0,T]$ to $0$, in probability, if and only if 
    \[
       \sum_{k=1}^{\infty} \int_0^T\abs{\tint \psi_k \left( \sigma_l( u_{\kappa,l})- u_{\kappa,l}\right)\nabla\eta}^2 \ra   0
    \]
    holds in probability. The latter can be obtained analogously to \eqref{eqn_kek} by using \eqref{regularity_of_eta} and \eqref{eqn_22*}. Therefore, we may pass to a further subsequence along which $\P$-a.s., for all $t\in [0,T]$, $\mathscr R_{\kappa,t}^{\mathrm{stoch}}(t) \ra 0$.

	\textit{The left-hand side.} For the left-hand side, we use that we may pass by \eqref{eq:GSU_convergences} to another subsequence, along which furthermore $ u_{\kappa,l} \to u $ in $L^r(\T^d)$, $\P\otimes \d t$-a.e., and for any $r<(1+2/d)(m+1)$. Since the latter means that this convergence holds $\P$-a.s., for a.e.\ $t\in [0,T]$, we conclude together with the previous steps that \eqref{eq:loc_mass_general} holds $\P$-a.s., for a.e.\ $t\in [0,T]$. Then using that both sides of \eqref{eq:loc_mass_general} are continuous time, it holds therefore also for all $t\in [0,T]$, finishing the proof in case that \eqref{eqn_add_q} holds.
	
	\textit{Removing the assumption \eqref{eqn_add_q}.} For the general case, we may follow the argument from \cite[Proposition 4.6]{GSU_SPME_Superlinear}: For each $R \in \N$ the truncated initial value $u_{0,R} = u_0\wedge R$ clearly satisfies \eqref{eqn_add_q}, so that \eqref{eq:loc_mass_general} is valid for the kinetic solution $u_R$ to \eqref{Eq_SPME_intro} with said initial value. Then \cite[Theorem 4.6]{Fehrman_Gess_ARMA} yields
	\begin{equation} \label{eqn_pathwise_contr}
		\| u_R-  u \|_{L^\infty (\Omega\times [0,T];L^1(\T^d )) } \le \|u_{0,R} - u_0\|_{L^1(\T^d)} \to 0,\qquad R\to\infty,
	\end{equation}
	which together with the estimates leading to \eqref{eq:GSU_convergences} yields  that also
	\begin{equation}\label{Eq_67}
		u_{R}\,\to \, u ,\qquad \text{in }L^p( \Omega ; L^r([0,T] \times \T^d)), \qquad r< \left( 1+ \frac{2}{d}\right)(m+1),
	\end{equation} cf.\ \cite[Corollary 4.3]{GSU_SPME_Superlinear}. 
	Thereby, a repetition of the above limiting argument shows that \eqref{eq:loc_mass_general} holds for $u$ as well.
    \end{proof}

\begin{lemma}
	\label{lemma:loc_energy_general}
 If  additionally $\eta_0 \ge 0$, then there exists 
  \begin{align}\label{ass_C}
 	C_0(m,\psi)<\infty,
 \end{align}
	so that for any $C\ge C_0$ 
and stopping time $\tau \le T$, we have
\begin{equation}
	\label{eq:energy_est_exp}
	\begin{split}
		&\e{\sup_{ t\in [0, \tau]}\eint{t} \eta(t)\, u(t)^2 + \int_0^\tau \eint{t} \eta  \abs{\nabla u^{\frac{m+1}{2}}}^2 }  \lesssim_{(m,\psi)} \, \int\eta_0  u_0^2 \, +\, \e{\int_0^\tau\eint{t} \abs{\Delta\eta} \, u^{m+1}}.
	\end{split}
\end{equation}
\end{lemma}
\begin{proof}As in the proof of Lemma \ref{lemma:loc_mass}, the starting point is an application of Proposition \ref{prop:Ito_product_rule}, but this time with $f(r)=r^2/2$ and with $\eta(t)$ replaced by $e^{-Ct}\eta(t)$. By the observations at the beginning of said proof, the assumptions of Proposition \ref{prop:Ito_product_rule} are satisfied, if we impose the additional assumption \eqref{eqn_add_q}. The result is the following identity that holds $\P$-a.s., for all $t\in [0,T]$:
	\begin{align}
		\frac{e^{-Ct}}{2}\tint \eta(t) \ukl^2(t) & = \frac{1}{2}\tint \eta_0 u_0^2 
        - \int_0^t e^{-Cs} \tint \nabla(\eta \ukl) \cdot \left( \nabla \ukl^{m}+\kappa \nabla\ukl \right) 
        \\&\quad -\frac{1}{2}\int_0^t e^{-Cs} \tint \nabla(\eta \ukl) \cdot\left(\Psi_1(\sigma_l'(\ukl))^2 \nabla \ukl  + \sigma_l'(\ukl )\sigma_l(\ukl)\Psi_2
		\right) \\
		&\quad 
		- \frac{C}{2}\int_0^t\eint{s}  \eta \ukl^2 
		- \frac{1}{4}\int_0^t e^{-Cs}\tint  \ukl^2 \Psi_2\cdot \nabla\eta  - \frac{1}{4}\int_0^t e^{-Cs} \tint \Psi_1\nabla(\ukl^2) \cdot \nabla\eta 
		\\& \quad 
		 + \sum_{k=1}^{\infty}\int_0^t e^{-Cs} \tint \ukl \psi_k \nabla\eta\cdot \nabla(\psi_k \sigma_l(\ukl))
		  + \frac{1}{2}\sum_{k=1}^{\infty} \int_0^t\eint{s}\eta \abs{\nabla(\psi_k\sigma_l(\ukl))}^2 \\
		&\quad - \frac{1}{2}\sum_{k=1}^{\infty} \int_0^t \eint{s} \psi_k\ukl^2 \nabla\eta \cdot \d \beta_k + \sum_{k=1}^{\infty} \int_0^t \eint{s} \psi_k \,  \sigma_l(\ukl)\nabla(\eta \ukl ) \cdot \d \beta_k.
	\end{align}
We also remark that two additional terms arose compared to \eqref{eqn_abc}, namely one due to the exponential time weight and one since this time $f'' \ne 0$. 
Expanding the latter leads to cancellations with the terms stemming from the Stratonovich correction in the equation for $u$ and thereby bringing us to
\begin{align}\label{eq:step1} 
    &  \frac{1}{2}\tint \eta_0 u_0^2 - \int_0^t e^{-Cs} \tint \nabla(\eta \ukl) \cdot \left( \nabla \ukl^{m}+\kappa \nabla\ukl\right) 
    \\&\quad -\frac{1}{2} \int_0^t e^{-Cs} \tint\nabla \eta \cdot \left(   
	\Psi_1 \ukl (\sigma_l'(\ukl))^2 \nabla \ukl  +  \ukl \sigma_l'(\ukl )\sigma_l(\ukl)\Psi_2
	\right) 
	\\&\quad 
	- \frac{C}{2}\int_0^t\eint{s}  \eta \ukl^2 
	- \frac{1}{4}\int_0^t e^{-Cs}\tint  \ukl^2 \Psi_2\cdot \nabla\eta  - \frac{1}{4}\int_0^t e^{-Cs} \tint \Psi_1\nabla(\ukl^2) \cdot \nabla\eta 
	\\& \quad +\int_0^t e^{-Cs} \tint 
	\nabla \eta \cdot \left(
	\Psi_1 \ukl \sigma_l'(\ukl) \nabla \ukl + \ukl \sigma_l (\ukl ) \Psi_2
	\right) 
	\\&\quad + \frac{1}{2} \int_0^t\eint{s} \eta\left(\sigma_l'(\ukl)\sigma_l(\ukl)\nabla \ukl\cdot \Psi_2 + \Psi_3  (\sigma_l(\ukl))^2\right) \\
	&\quad - \frac{1}{2}\sum_{k=1}^{\infty}\int_0^t \eint{s} \psi_k\ukl^2 \nabla\eta \cdot \d \beta_k + \sum_{k=1}^{\infty} \int_0^t \eint{s} \psi_k \,  \sigma_l(\ukl)\nabla(\eta \ukl ) \cdot \d \beta_k 
\end{align}
for the above. Defining 
\begin{align}
\nu_l(r) = \int_0^r  \sigma_l (s) \,\d s,
\end{align}
we may integrate by parts in the emerging term as well as the last stochastic integral, resulting in 	\begin{align}&
	\int_{\T^d} \eta \sigma_l'(\ukl)\sigma_l(\ukl)\nabla \ukl\cdot \Psi_2 = - \int_{\T^d}\frac{1}{2} (\sigma_l (\ukl ))^2  \diver  (\eta \Psi_2) ,
	\\&
	\int_{\T^d} \psi_k \,  \sigma_l(\ukl)\nabla(\eta \ukl ) = \int_{\T^d} \psi_k \ukl  \sigma_l(\ukl)\nabla \eta   -  \int_{\T^d}  \nu_l(\ukl)\nabla (\eta \psi_k)
	.
	\end{align}
	Introducing again remainder terms, this time defined  as
	\begin{align}\begin{split}\label{eqn_rdeterrich}
		\mathscr R^{\mathrm{det}}_{\kappa,l}(t) & =  \frac{\kappa}{2}\int_0^t \eint{s} \ukl^2 \Delta  \eta 
        \\& \qquad +\int_0^t \eint{s} \Psi_1\nabla\eta\cdot  \left( \ukl\sigma_l'(\ukl)- \frac{1}{2}\ukl(\sigma_l'(\ukl))^2-\frac{1}{2}\ukl\right)\nabla\ukl \\ 
		& \qquad + \int_0^t \nabla\eta\cdot \left( \ukl \sigma_l(\ukl) - \frac{1}{2}\ukl\sigma_l'(\ukl) \sigma_l(\ukl) - \frac{1}{4} \ukl^2 -\frac{1}{4} (\sigma_l (\ukl ))^2  \right)  \Psi_2, 
		\end{split} \\
		\mathscr R^{\mathrm{stoch}}_{\kappa,l}(t) &= \int_0^t\eint{s} \psi_k \left( \ukl\sigma_l(\ukl) - \frac{1}{2}\ukl^2- \nu_l(\ukl) \right) \nabla\eta\cdot\d \beta_k,
	\end{align}  that $\frac{1}{2}\eint{t} \eta(t) \ukl(t)^2$ equals
	\eqref{eq:step1} may be equivalently written by gathering like terms as
\begin{equation}\label{eq:loc_en_comp_1-o}\begin{split}
			&\frac{1}{2}\eint{t} \eta(t) \ukl(t)^2 + m\int_0^t\eint{s} \eta\ukl^{m-1}\abs{\nabla\ukl}^2+ \kappa\int_0^t\eint{s} \eta \abs{\nabla \ukl}^2 \\	& \quad +\frac{1}{2}\int_0^t\eint{s} \eta \left( C\ukl^2-\Psi_3 (\sigma_l(\ukl))^2+ \frac{1}{2} \diver(\Psi_2)(\sigma_l (\ukl ))^2 \right)\\
			&= \frac{1}{2}\tint \eta_0 u_0^2 + \frac{m}{m+1}\int_0^t\eint{s}  \ukl^{m+1}\Delta\eta -\sum_{k=1}^{\infty}\int_0^t\eint{s} \eta \nu_l(\ukl) \nabla\psi_k\cdot\d \beta_k 
		 + \mathscr R^{\mathrm{det}}_{\kappa,l}(t) + \mathscr R^{\mathrm{stoch}}_{\kappa,l}(t)
			.
	\end{split}\end{equation}
Using that $\sigma_{l}^2(r)\le r^2$  by \eqref{Eq_21}, that $\eta \ge 0$ by Lemma \ref{lemma:representation_formula},  our assumption on $\eta_0$, and defining
\begin{align}C_1(\psi) := \sup_{x\in\T^d} \Bigl( \Psi_3  -\frac{1}{2} \diver(\Psi_2)\Bigr),
\end{align}
we deduce the estimate 
	\begin{equation}\label{eq:loc_en_comp_1}\begin{split}
		&\frac{1}{2}\eint{t} \eta(t) \ukl(t)^2 + m\int_0^t\eint{s} \eta\ukl^{m-1}\abs{\nabla\ukl}^2+ 
        \\& \hspace{1.5cm} + \kappa\int_0^t\eint{s} \eta \abs{\nabla \ukl}^2+\frac{C-C_1(\psi)}{2}\int_0^t\eint{s} \eta \ukl^2
		\\&\quad\le 
		\frac{1}{2}\tint \eta_0 u_0^2 + \frac{m}{m+1}\int_0^t\eint{s}  \ukl^{m+1}\Delta\eta 
        \\&\hspace{1.5cm} -\sum_{k=1}^{\infty}\int_0^t\eint{s} \eta \nu_l(\ukl) \nabla\psi_k\cdot\d \beta_k + \mathscr R^{\mathrm{det}}_{\kappa,l}(t) + \mathscr R^{\mathrm{stoch}}_{\kappa,l}(t)
		,
\end{split}\end{equation}
$\P$-a.s., for all $t\in [0,T]$. 
To complete the proof, we proceed in the following three steps.

\textit{Limiting procedure $\kappa\ra 0$ and $l\to \infty$.}We wish to establish a limiting version of \eqref{eq:loc_en_comp_1} for the kinetic solution $u$ to \eqref{Eq_SPME_intro}. To this end, we take the the supremum over $t\in [0,\tau]$ in the latter inequality to deduce that 
\begin{align}	\begin{split}\label{eqn_4}
		&  \sup_{t\in [0,\tau] } \eint{t} \eta(t) u_{\kappa,l}(t)^2 + \int_0^\tau \eint{t} \eta \abs{\nabla u_{\kappa,l}^{\frac{m+1}{2}}}^2 +\frac{C-C_1(\psi)}{2}\int_0^\tau \eint{t} \eta u_{\kappa,l}^2 
		 \\ & \quad 
		\lesssim_m  \frac{1}{2}\int\eta_0  u_0^2 	 + \int_0^\tau \eint{t}  u_{\kappa,l}^{m+1}|\Delta\eta|  + \sup_{t\in [0,\tau]} \biggl| \sum_{k=1}^{\infty} \, \int_0^t\eint{s} \eta  \nu_l (u_{\kappa,l}) \nabla\psi_k \cdot \d \beta_k \biggr|\\ & \qquad  + \sup_{t\in [0,\tau]} | \mathscr R^{\mathrm{det}}_{\kappa,l}(t)| + \sup_{t\in [0,\tau]}  | \mathscr R^{\mathrm{stoch}}_{\kappa,l}(t)|.
	\end{split}
\end{align}
To take $\P$-a.s. the limit on both sides, we use Fatou's Lemma and \eqref{Eq_convergences}, which  implies up to passing to a further subsequence that
\begin{align}&
	\sup_{t\in [0,\tau] } \eint{t} \eta(t) u(t)^2 + \int_0^\tau \eint{t} \eta \abs{\nabla u^{\frac{m+1}{2}}}^2 +\frac{C-C_1(\psi)}{2}\int_0^\tau \eint{t} \eta u_{\kappa,l}^2  \\&\quad \le \liminf_{(\kappa,l)\to (0,\infty)}\biggl(\sup_{t\in [0,\tau] } \eint{t} \eta(t) u_{\kappa,l}(t)^2 + \int_0^\tau \eint{t} \eta \abs{\nabla u_{\kappa,l}^{\frac{m+1}{2}}}^2  +\frac{C-C_1(\psi)}{2}\int_0^\tau \eint{t} \eta u^2 \biggr).
\end{align}
It remains to argue about the right-hand side of \eqref{eqn_4}. The term involving the initial data is unaffected by said limit. Passing to a further subsequence along which also \eqref{eqn_22*} holds $\P$-a.s., we find that 
\begin{align}\int_0^\tau \eint{t} u_{\kappa,l}^{m+1} |\Delta\eta| \to \int_0^\tau \eint{t}  u^{m+1}|\Delta\eta| ,
\end{align}
$\P$-a.s., due to the regularity \eqref{regularity_of_eta} of $\eta$. We claim that also 
\begin{align}\sup_{t\in [0,\tau]} \biggl| \sum_{k=1}^{\infty} \, \int_0^t\eint{s} \eta  \nu_l (u_{\kappa,l}) \nabla\psi_k \cdot \d \beta_k \biggr| \to \sup_{t\in [0,\tau]} \biggl| \sum_{k=1}^{\infty} \, \int_0^t\eint{s} \eta  u^2\nabla\psi_k \cdot \d \beta_k \biggr|,
\end{align}
$\P$-a.s.,
up to taking another subsequence, for which it suffices to show 
\begin{align} \sum_{k=1}^{\infty}\int_0^T\abs{\tint \left(\eta  \nu_l (u_{\kappa,l})  -  \eta  u^2\right) \nabla\psi_k }^2 \ra   0,
\end{align}
$\P$-almost surely. But this follows from the upper bound 
\begin{align}
	\| \eta \|_{L^\infty([0,T] \times \T^d)}^2\biggl( \int_0^T \biggl(\sum_{k=1}^{\infty}  \int_{\T^d}
	|\nu_l(u_{\kappa,l}) - u^2  |^{1-\delta/2} |\nabla \psi_k|^2 \biggr)^{1/\delta}\biggr)^\delta  \biggl(\int_0^T \biggl(\int_{\T^d}| \nu_l (u_{\kappa,l}) - u^2|^{1+\delta/2}\biggr)^{\frac{1}{1-\delta}}\biggr)^{1-\delta},
\end{align}
for sufficiently small $\delta>0$, because of \eqref{eqn_ass_noise}, and since \eqref{eqn_convergence_by_boundedness_LinftyL2}, \eqref{regularity_of_eta}, and \eqref{eqn_22*} hold $\PP$-almost-surely. With the same arguments and an additional integration by parts in the second term on the right-hand side of \eqref{eqn_rdeterrich}, one sees that, $\P$-a.s.,
\[
\sup_{t\in [0,\tau]} | \mathscr R^{\mathrm{det}}_{\kappa,l}(t)| + \sup_{t\in [0,\tau]}  | \mathscr R^{\mathrm{stoch}}_{\kappa,l}(t)| \ra 0.
\]
Thereby, we conclude that
\begin{equation}
	\label{eq:loc_energy_general}
	\begin{split}
		&  \sup_{t\in [0,\tau] } \eint{t} \eta(t) u(t)^2 + \int_0^\tau \eint{t} \eta \abs{\nabla u^{\frac{m+1}{2}}}^2 +\frac{C-C_1(\psi)}{2}\int_0^\tau \eint{t} \eta u^2   \\ & \quad
	\lesssim_m  \frac{1}{2}\int\eta_0  u_0^2 	 + \int_0^\tau \eint{t} |\Delta\eta| \, u^{m+1} + \sup_{t\in [0,\tau]} \biggl| \sum_{k=1}^{\infty} \, \int_0^t\eint{s} \eta  u^2\nabla\psi_k \cdot \d \beta_k \biggr|,
	\end{split}        
\end{equation}
holds $\P$-almost surely.

\textit{Estimate in expectation.} To take the expectation in the above, we introduce the stopping times 
\begin{align}\label{defi_tau_N}
	\tau_N = \inf \bigl\{ t\in [0,\tau] \big| \| u \|_{L^\infty(0,t ; L^2(\T^d) ) }  + \| \eta \|_{ L^\infty(0,t ; W^{2,\infty}(\T^d))}
	\ge  N  \bigr\}\wedge \tau ,
\end{align}
satisfying $\P( \tau_N =\tau) \to 1$ by \eqref{eqn_boundedness_LinftyL2_subst}, Fatou's lemma  and \eqref{regularity_of_eta}. Applying \eqref{eq:loc_energy_general} with $\tau$ replaced by $\tau_N$ yields, after taking the expectation and applying the Burkholder--Davis--Gundy inequality, that
\begin{align}\begin{split}\label{eqn_wuppertal}&\e{\sup_{0\leq t\leq \tau_N}\eint{t} \eta(t)\, u(t)^2 + \int_0^{\tau_N} \eint{t} \eta  \abs{\nabla u^{\frac{m+1}{2}}}^2 
+\frac{C-C_1(\psi)}{2}\int_0^{\tau_N} \eint{t} \eta u^2 	
}  
	\\&\quad  \lesssim_{m} \, \int\eta_0  u_0^2 \, +\, \e{\int_0^{\tau_N}\eint{t} \abs{\Delta\eta} \, u^{m+1}} 
	+ \E\biggl[\biggl(
	\sum_{k=1}^{\infty} \int_0^{\tau_N} e^{-2Ct} \abs{
\int_{\T^d} \eta  u^2\nabla\psi_k	
}^2 \biggr)^{1/2}
	\biggr].
\end{split}
\end{align}
Regarding the latter term, we use  
\begin{align}& \E\biggl[\biggl(
	\sum_{k=1}^{\infty}
	 \int_0^{\tau_N} e^{-2Ct} \abs{
		\int_{\T^d} \eta  u^2\nabla\psi_k	
	}^2 \biggr)^{1/2}
	\biggr]
	\\&\quad 
	\lesssim \e{  \left(\sup_{t\in [0,\tau_N] } e^{-Ct}
	\int_{\T^d} \eta u^2 \right)^{1/2} \left( \sum_{k=1}^{\infty}
	\int_0^{\tau_N} e^{-Ct} 
		\int_{\T^d} \eta  u^2|\nabla\psi_k|^2	
	  \right)^{1/2}},
\end{align}
which may be absorbed in the left-hand side of \eqref{eqn_wuppertal}, if $C_0$ in \eqref{ass_C} is large enough. This implies \eqref{eq:energy_est_exp} with $\tau$ replaced by $\tau_N$, and we may let $N\to\infty$ to deduce  \eqref{eq:energy_est_exp} with $\tau $ by Fatou's lemma. 

\textit{Removing the assumption \eqref{eqn_add_q}.}
If $u_0\notin L^{m+1}(\T^d)$, we may use the same line of argument as in the proof of Lemma \ref{lemma:loc_mass}, where we argue first that \eqref{eq:energy_est_exp} holds with $\tau$ replaces by $\tau_N$ as defined in \eqref{defi_tau_N}: We consider for $R \in \N$ the truncated initial value $u_{0,R} = u_0\wedge R$, so that  \eqref{eq:energy_est_exp}, with $(u,\tau)$ replaced by $(u_R,\tau_N)$ is valid, where $u_R$ is the kinetic solution to \eqref{Eq_SPME_intro} started from $u_{0,R}$. But then the convergence \eqref{Eq_67} and that $\mathbf{1}_{[0,\tau_N]} \Delta \eta \in L^\infty(\Omega\times [0,T]\times \T^d)$ yields the same for $(u,\tau_N)$  by means of Fatou's lemma. Another application of Fatou let's us conclude the same for $(u,\tau)$, finishing the proof. 
\end{proof}

\section{The proof of Theorem \ref{thm:lower} and Corollary \ref{col:lower}} \label{SS:proof_upper}

This section is dedicated to the proofs of Theorem \ref{thm:lower} and Corollary \ref{col:lower}. To begin, we restrict our analysis to probabilistic intervals on which the flow remains controlled by defining the stopping time
\begin{align}\label{eq:sigma_gamma}
    \sigma_\gamma =\inf  \{ t\in [0,T] : \|\phi - \id \|_{C([0,t];C^2(\T^d))} + 
    \|\phi^{-1} - \id \|_{C([0,t];C^2(\T^d))}
    >\gamma\},
\end{align}
where $\gamma\in (0,\infty)$ is a parameter of our choosing. We observe due to Lemma \ref{lemma_flows} that $\sigma_\gamma>0$, $\P$-a.s., for each $\gamma>0$. We will argue on $[0,\sigma_\gamma]$ when interested in short-time behavior of kinetic solutions to \eqref{Eq_SPME_intro}.

For given $x_0\in \T^d$, we introduce the localizing function $G:\Omega\times [0,T]\times [0,L]\ra \R_{\geq 0}$, defined by
\begin{align}\label{eqn_G}
    G(t,r) := 
    \sup_{0\leq t'\leq t\wedge \sigma_\gamma} e^{-Ct'}\int_{\phi_{t'}(B_{L-r}(x_0))}  u^2(t'),
\end{align}
where $\phi$ is the  stochastic flow of diffeomporphisms induced by \eqref{eqn_ODE_strat} provided by Lemma \ref{lemma_flows}. In particular, because $\sigma_\gamma>0$ $\P$-a.s., $G(t,L-r)=0$ for some $t>0$ if and only if $T_w^\phi(B_r(x_0))>0$. 

The assertions of Theorem \ref{thm:lower} are proved by the following Stampacchia-type inequality for $G$ obtained by iterating the randomly localized energy estimate from Lemma \ref{lemma:loc_energy_general}.
In what follows, we utilize the set
\[
    \lltriangle_L =\left\{ (s,\delta) \in (0,L)^2 \,:\, \delta\in (0,L-s] \right\}.
\]
We also recall that \eqref{eqn_ID}--\eqref{eqn_ass_noise} are assumed throughout and that $u$ denotes the unique kinetic solution to \eqref{Eq_SPME_intro}.
\begin{prop}
\label{prop:stampacchia}    If $x_0\in \T^d\setminus \supp(u_0)$ and $G$ is the quantity \eqref{eqn_G}, then
    we have 
    \begin{equation}
    \label{eq:stampachhia_type}
        \e{G(\tau,s+\delta)} \lesssim_{(m,\psi,\gamma)} \, \int_{ B_{L-s(x_0)}} u_0^2 \, + \, \e{\left(\int_0^\tau e^{\frac{(m-1)Ct}{2}} \, \d t \right) \, \frac{G(\tau,s)^{\frac{m+1}{2}}}{\delta^{2+\frac{d(m-1)}{2}}}}
    \end{equation}
    for each $(s,\delta) \in \lltriangle_{L}$, each stopping time $\tau\leq T$, and each constant $C$ larger than the constant $C_0(m,\psi)$ from \eqref{ass_C}.
\end{prop}

To prove Proposition \ref{prop:stampacchia}, we will use the following lemma, which is proven as \cite[Lemma 2.1]{Utley_TFPME}.

\begin{lemma}
    \label{lemma:iteration_trick} Let $L>0$ and consider non-negative mappings $\mathscr V:[0,L]\ra \R_{\geq0}$, along with non-negative functions $\mathscr E$, $\mathscr A$, $\mathscr U$, and $\mathscr F$ defined on $[0,L]\times \R_+$. If the following conditions are fulfilled:
    \begin{enumerate}[label=(\roman*)]
        \item The maps $\mathscr E,\mathscr A,\mathscr U,\mathscr F$ are non-increasing in the first variable;
        \item for any $\epsilon >0$, there exists $C_\epsilon>0$ such that for each $(s,\delta)\in \lltriangle_L$, we have
            \begin{equation}\label{eq-iteration-lemma-general}
                \begin{split}
                    \mathscr V(s+\delta) + \mathscr E(s+\delta,\delta) & \lesssim \, \mathscr A\left(s+\frac{\delta}{2},\frac{\delta}{2}\right) + \mathscr U\left(s+\frac{\delta}{2},\frac{\delta}{2}\right), \\
                    \mathscr U(s+\delta,\delta) & \lesssim \,  \epsilon \mathscr U\left(s+\frac{\delta}{2},\frac{\delta}{2}\right) + \epsilon \mathscr E\left(s+\frac{\delta}{2},\frac{\delta}{2} \right) + C_\epsilon \mathscr F\left(s+\frac{\delta}{2}, \frac{\delta}{2} \right);
                \end{split}
            \end{equation}
        \item there exist constants $E,A,U,F\in \R$ such that for any $\eta \in (0,1)$, we have 
            \begin{equation}\label{eq-iteration-lemma-condition}
                \begin{split}
                    \mathscr E(s, \eta\delta) \leq \eta^E \, \mathscr E(s, \delta), \\
                    \mathscr A(s, \eta\delta) \leq \eta^A \, \mathscr A(s, \delta), \\
                    \mathscr U(s, \eta\delta) \leq \eta^U \, \mathscr U(s, \delta), \\
                    \mathscr F(s, \eta\delta) \leq \eta^F \, \mathscr F(s, \delta),
                \end{split}
            \end{equation}
        for each $(s,\delta)\in [0,L]\times \R_+$;
    \end{enumerate}
    then for any $(s,\delta)\in \lltriangle_L$, we have
    \[
        \mathscr V(s+\delta)  \lesssim \, \mathscr A\left(s,\delta\right) \, + \, \mathscr F\left(s,\delta\right).
    \]
\end{lemma}

\begin{proof}[Proof of Proposition \ref{prop:stampacchia}]
Let $x_0\in\T^d\setminus \supp(u_0)$ and consider the quantity $G$ from \eqref{eqn_G}. The argument follows from Lemma \ref{lemma:iteration_trick} together with the localized energy estimate \eqref{eq:energy_est_exp} and the version of the Gagliardo--Nirenberg inequality established in Corollary \ref{app-L-2}. Compared  to a similar argument from  \cite[Lemma 5.1]{GSU_SPME_Superlinear} with a temporally constant test function, additional contributions of the flow are controlled by the presence of the stopping time $\sigma_\gamma$ in the definition of $G$.

Let us start with some preliminaries. For a given pair $(s',\delta') \in \lltriangle_{L}$, we consider smooth cut-off functions $\varphi_{s',\delta'}$ taking values in $[0,1]$, such that 
    \begin{enumerate}[label=(\roman*)]
        \item \label{prop_i} $\supp(\vp_{s',\delta'})= B_{L-s'}(x_0)$,
        \item $\vp_{s',\delta'} \equiv 1$ in $B_{L-s'-\delta'}(x_0)$,
        \item and we have $\abs{D^j \vp_{s',\delta'}}\lesssim 1/(\delta')^j$ for $j=1,2$, where the constants depend on neither $s'$ nor $\delta'$.
    \end{enumerate}
{We observe that the property \ref{prop_i} ensures that $\vp_{s',\delta'}$ has compact effective support, cf.\ Definition \ref{defi_CES}.}
We let $\eta_{s',\delta'}$ be the solution to the stochastic transport equation \eqref{eqn_tpe} started from $\eta_{s',\delta'}(0) = \vp_{s',\delta'}$ provided by Lemma \ref{lemma:representation_formula}. The representation formula \eqref{eqn_repr_by_stoch_flow} tells us that $\P$-a.s.
\begin{align}\label{eqn_above}
	\eta_{s',\delta'}(t,x) = (\vp_{s',\delta'}\circ \phi_t\inv )(x), \qquad (t,x)\in [0,T]\times\T^d.
\end{align}
Thereby, we may translate the above properties into that $\PP$-a.s.
\begin{enumerate}[label=(\roman*')]
        \item \label{eta_1} $\supp(\eta_{s',\delta'}(t))= \phi_t(B_{L-s'}(x_0))$, for $t\in [0,T]$,
        \item \label{eta_2} $\eta_{s,\delta} \equiv 1$ in $\phi_t(B_{L-s'-\delta'}(x_0))$, for $t\in [0,T]$,
        \item \label{eta_3} $\abs{\Delta \eta_{s',\delta'}(t)} \lesssim (1+\gamma)^2 \delta^{-2} + \gamma \delta\inv$ for  $t\in [0,  \sigma_\gamma]$.
\end{enumerate}    
For the last item we used \eqref{eqn_formula} to compute that
\begin{align*}
    \Delta \eta_{s',\delta'} = \mathrm{tr}\left( (D\phi\inv_t)^\top (D^2\varphi_{s',\delta'}\circ \phi_t\inv) D\phi_t\inv \right) + \Delta\phi_t\inv \cdot (\nabla \vp_{s',\delta'} \circ\phi_t\inv).
\end{align*}
Taking absolute values and recalling the definition \eqref{eq:sigma_gamma} of $\sigma_\gamma$ gives the  inequality stated in \ref{eta_3}. Lemma \ref{app-L-3} and \eqref{eqn_above} imply moreover that also $\eta_{s',\delta'}$ has $\P$-a.s.\ compact effective support, for all $t\in [0,T]$.

Fix now an arbitrary pair $(s,\delta)\in \lltriangle_L$ and a stopping time $\tau\le T$. Set $s'=s+\frac{\delta}{2}$ along with $\delta'=\frac{\delta}{2}$. The energy estimate \eqref{eq:energy_est_exp}, with stopping time $\tau\wedge \sigma_\gamma$ and random localization $\eta_{s',\delta'}$, combined with the above properties yields 
\begin{align}\label{eq:stamp_proof_1}
    &\e{\sup_{0\leq t\leq \tau\wedge \sigma_\gamma}\Eint{t}{\phi_t(B_{L-s-\delta}(x_0))} u(t)^2 + \int_0^{\tau\wedge \sigma_\gamma} \Eint{t}{\phi_t(B_{L-s-\delta}(x_0))}  \abs{\nabla u^{\frac{m+1}{2}}}^2 } \\
    &\hspace{2cm} \lesssim_{(m,\psi,\gamma)} \, \int_{B_{L-s-\frac{\delta}{2}}(x_0)} u_0^2 \, +\, \left(\frac{4}{\delta^2}+\frac{2}{\delta}\right) \, \e{\int_0^{\tau\wedge \sigma_\gamma}\Eint{t}{\phi_t(B_{L-s-\frac{\delta}{2}}(x_0))} u^{m+1}} \\
    &\hspace{2cm} \lesssim \, \int_{B_{L-s-\frac{\delta}{2}}(x_0)} u_0^2 \, +\, \frac{4}{\delta^2} \, \e{\int_0^{\tau\wedge \sigma_\gamma}\Eint{t}{\phi_t(B_{L-s-\frac{\delta}{2}}(x_0))} u^{m+1}}.
\end{align}
In the last estimate, we used that $\delta\le L$. As we aim to apply Lemma \ref{lemma:iteration_trick}, this motivates the following definitions:
\begin{align*}
        \mathscr V(s) &:=  \e{\sup_{0\leq t\leq \tau\wedge\sigma_\gamma}\Eint{t}{\phi_t(B_{L-s}(x_0))} u(t)^2},\\
        \mathscr E(s,\delta) &:= \e{\int_0^{\tau\wedge \sigma_\gamma} \Eint{t}{\phi_t(B_{L-s}(x_0))}  \abs{\nabla u^{\frac{m+1}{2}}}^2 },\\
        \mathscr U(s,\delta) &:= \delta^{-2}\e{\int_0^{\tau\wedge \sigma_\gamma}\Eint{t}{\phi_t(B_{L-s}(x_0))} u^{m+1}},  \\
        \mathscr A(s,\delta) &:= \int_{B_{L-s}(x_0)} u_0^2\, .
\end{align*}
Indeed, the estimate \eqref{eq:stamp_proof_1}, when written in terms of these new functions, reads 
\[
    \mathscr V(s+\delta) + \mathscr E(s+\delta,\delta)  \lesssim_{(m,\psi,\gamma)} \, \mathscr A\left(s+\frac{\delta}{2},\frac{\delta}{2}\right) + \mathscr U\left(s+\frac{\delta}{2},\frac{\delta}{2}\right),
\]
i.e., the same inequality as the first estimate in \eqref{eq-iteration-lemma-general} required to apply Lemma \ref{lemma:iteration_trick}. To derive the second inequality in \eqref{eq-iteration-lemma-general}, we wish to apply a homogeneous version of the Gagliardo--Nirenberg inequality from Corollary \ref{app-L-2} to the function
\[
    v(t) := \eta_{s+\frac{\delta}{2},\frac{\delta}{2}}(t) \, u(t)^{\frac{m+1}{2}}.
\]
We note that this is permitted since the compact effective support property of $\eta_{s+\frac{\delta}{2},\frac{\delta}{2}}$ carries over to $v$. With this in mind, we utilize the properties of $\eta_{s',\delta'}$ {for $s' = s+\frac{\delta}{2}$ and $\delta' =\frac{\delta}{2}$} listed above, Corollary \ref{app-L-2}, and Young's inequality to obtain
\begin{align*}
    \mathscr U(s+\delta,\delta) &  := \frac{1}{\delta^2}\e{\int_0^{\tau\wedge \sigma_\gamma}\Eint{t}{\phi_t(B_{L-s-\delta}(x_0))} u^{m+1}}
    \\& \overset{\ref{eta_1},\,\ref{eta_2}}{\leq} \frac{1}{\delta^2} \e{\int_0^{\tau\wedge \sigma_\gamma}\Eint{t}{\T^d} \eta_{s+\frac{\delta}{2},\frac{\delta}{2}}^2 \, u^{m+1}} 
    \\& = \frac{1}{\delta^2} \e{\int_0^{\tau\wedge \sigma_\gamma}\Eint{t}{\T^d} v^2} 
    \\& \overset{\eqref{eq-GN-1}} {\lesssim_{m}}\, \frac{1}{\delta^2}\e{ \int_0^{\tau\wedge \sigma_\gamma} e^{-Ct} \left(\tint \abs{\nabla v}^2\right)^{\theta}\left( \tint v^{\frac{4}{m+1}} \right)^{(1-\theta)\left( \frac{m+1}{2} \right)} } 
    \\& \leq \, \epsilon\e{\int_0^{\tau\wedge \sigma_\gamma} e^{-Ct} \tint \abs{\nabla v}^2}
    + C_\epsilon \delta^{-\frac{2}{1-\theta}} \e{\int_0^{\tau\wedge \sigma_\gamma} e^{-Ct} \left( \tint v^{\frac{4}{m+1}} \right)^{ \frac{m+1}{2} }},
\end{align*}
where 
\[
    \theta = \frac{d(m-1)}{4+d(m-1)}.
\]
Recalling the properties \ref{eta_1} and \ref{eta_3} of $\eta_{s+\frac{\delta}{2},\frac{\delta}{2}}$, we see that 
\begin{align*}
    \e{\int_0^{\tau\wedge \sigma_\gamma} e^{-Ct} \tint \abs{\nabla v}^2} & \lesssim_{\gamma} \, \mathscr E\left( s+\frac{\delta}{2},\frac{\delta}{2} \right) + \mathscr U\left( s+\frac{\delta}{2},\frac{\delta}{2}\right) 
    \\   \e{\int_0^{\tau\wedge \sigma_\gamma} e^{-Ct} \left( \tint v^{\frac{4}{m+1}} \right)^{ \frac{m+1}{2} }} & \leq  \e{\int_0^{\tau\wedge \sigma_\gamma} e^{-Ct} \left( \int_{\phi_t(B_{L-s-\frac{\delta}{2}}(x_0))} u^2 \right)^{ \frac{m+1}{2} }}.
\end{align*}
Plugging this into the prior estimate, we find
\begin{align}\label{eq:stamp_proof_2}
    \mathscr U(s+\delta,\delta) &  \lesssim_{(m,d,\gamma)} \, \epsilon   \mathscr U\left(s+\frac{\delta}{2},\frac{\delta}{2}\right) + \epsilon \mathscr E\left(s+\frac{\delta}{2},\frac{\delta}{2} \right) 
   \\& \hspace{2cm} + \, C_\epsilon \,  \delta^{-\frac{2}{1-\theta}} \e{\int_0^{\tau\wedge \sigma_\gamma} e^{-Ct} \left( \int_{\phi_t(B_{L-s-\frac{\delta}{2}}(x_0))} u^2 \right)^{ \frac{m+1}{2} }}.
\end{align}
Once we define
\[
    \mathscr F(s,\delta) := \delta^{-\frac{2}{1-\theta}}\, \e{\left( \int_0^{\tau\wedge \sigma_\gamma} e^{\frac{(m-1)Ct}{2}} \right)\left( \sup_{0\leq t\leq \tau\wedge \sigma_\gamma} \Eint{t}{\phi_t(B_{L-s}(x_0))} u^2 \right)^{\frac{m+1}{2}}},
\]
then we infer from \eqref{eq:stamp_proof_2} that 
\[
    \mathscr U(s+\delta,\delta)   \lesssim_{(m,\gamma)} \, \epsilon   \mathscr U\left(s+\frac{\delta}{2},\frac{\delta}{2}\right) + \epsilon \mathscr E\left(s+\frac{\delta}{2},\frac{\delta}{2} \right)  + C_\epsilon \mathscr F\left(s+\frac{\delta}{2},\frac{\delta}{2}\right),
\]
i.e., the second estimate in \eqref{eq-iteration-lemma-general} required to apply Lemma \ref{lemma:iteration_trick}. Naturally the maps $\mathscr V, \mathscr A,\mathscr E,\mathscr U,\mathscr F$ are non-increasing in the first variable. Moreover, the conditions \eqref{eq-iteration-lemma-condition} hold with $E=A=0$, $U = -2$, and 
\[
    F = -\frac{2}{1-\theta} = -2-\frac{d(m-1)}{2}.
\]
Applying Lemma \ref{lemma:iteration_trick}, the result is proven, since by definition we have
\begin{align*}
    \mathscr V(s) &= \e{G(\tau,s)} 
    \\ \mathscr F(s,\delta) &= \delta^{-2-\frac{d(m-1)}{2}} \e{\left( \int_0^{\tau\wedge \sigma_\gamma} e^{\frac{Ct(m-1)}{2}} \right) G(\tau,s)^{\frac{m+1}{2}}}.
\end{align*}
\end{proof}

We are in the position to prove Theorem \ref{thm:lower} and Corollary \ref{col:lower} on finite speed of propagation and sufficient conditions for parabolic waiting time phenomena of \eqref{Eq_SPME_intro}. 

\begin{proof}[Proof of Theorem \ref{thm:lower}]
    {
    \emph{Ad \ref{thm:point_1}.} Let $x_0\in \T^d\setminus\supp(u_0)$ and $r\in (0,R_0\wedge L)$. Fix an arbitrary $\gamma>0$. It is enough to show that $T_w^\phi(B_r(x_0)) \wedge \sigma_\gamma$ is almost-surely positive. Indeed, in view of the discussion preceding Proposition \ref{prop:stampacchia}, $\PP(\sigma_\gamma>0) =1$, and so 
    \[
        \PP\left(T_w^\phi(B_r(x_0))\wedge \sigma_\gamma >0\right) = \PP\left(T_w^\phi(B_r(x_0))>0\right).
    \]
    From this point on the argument is nearly identical to that of \cite[Theorem 2.6, Theorem 5.3]{GSU_SPME_Superlinear}, which is based on the the approach of \cite{Fischer_Grun_FSOP}. We recapitulate the main details. To ease notation in what follows, define 
    \[
        Z(t) := \int_0^t e^{\frac{(m-1) Cs}{2}} \, \d s.
    \]
    Note that $Z$ is continuous, strictly increasing, and $Z_0=0$. With this notation, the estimate \eqref{eq:stampachhia_type} now reads
    \[
        \e{G(\tau,s+\delta)} \lesssim_{(m,\psi,\gamma)} \, \int_{{B_{L-s}(x_0)}} u_0^2 \, + \, \e{Z(\tau)\, \frac{G(\tau,s)^{\frac{m+1}{2}}}{\delta^{2+\frac{d(m-1)}{2}}}}.
    \]
    
    In order to show that $\PP\left(T_w^\phi(B_r(x_0))>0\right)=1$, we will show that the event $\{T_w^\phi(B_r(x_0))=0\}$ has probability zero. For this statement it suffices to prove that
    \[
        \lim_{t\ra 0} \, \PP\left( G(t,L-r) >0 \right) \, =\, 0.
    \]
    One proceeds by decomposing the event $\{G(t,L-r)>0\}$ as follows. Set $r_0 =L- (R_0\wedge L)$ and 
    \[
    r_n= L- r - \frac{(R_0\wedge L)-r}{2^n},
    \]
    Further, take a strictly decreasing sequence $(\mu_n)_{n\in \N}$ such that $\mu_n \downarrow 0$. We leverage these sequences to obtain the decomposition (cf. \cite[p. 21]{Fischer_Grun_FSOP})
    \[
        \left\{ G(t,L-r) >0 \right\} \subset \bigcup_{n=1}^\infty\left\{ G(\tau_n,r_n) > \mu_n \right\},
    \]
    where the $\tau_n$ are the following stopping times:
    \[
        \tau_1 = t, \quad \tau_n = \tau_{n-1} \wedge \inf\left\{\xi\in [0,T], \, G(\xi,r_{n-1})\geq \mu_{n-1}\right\}.
    \]
    From here, we may estimate the probabilities using Markov's inequality along with \eqref{eq:stampachhia_type}, which yields for all $n\in \N$
    \[
        \PP\left( G(\tau_n,r_n) >\mu_n\right) \, \lesssim_{(m,\gamma,\psi)} \, \frac{2^{\left(2+\frac{d}{2}(m-1)\right)n}}{((R_0\wedge L)-r)^{2+\frac{d}{2}(m-1)}}\e{Z(\tau_n)G(\tau_n,r_{n-1})^{\frac{m+1}{2}}} .
    \]
    For $n\geq 2$, the latter may be estimated further using $G(\tau_n,r_{n-1})\leq \mu_{n-1}$. When $n=1$, we infer from \eqref{eq:energy_est_exp} with $\eta=\mathbf{1}_{\T^d}$ that $\e{G(t,R_0)} \lesssim_{(m,d)} \, \norm{u_0}_{L^2}^2$. Altogether then, we have 
    \[
        \PP\left( G(t,r)>0\right) \, \lesssim_{(m,\gamma,\psi)} \, \frac{Z(t)}{((R_0\wedge L)-r)^{2+\frac{d}{2}(m-1)}} \left( \frac{\norm{u_0}_{L^2}^{m+1}\, 2^{\left(2+\frac{d}{2}(m-1)\right)}}{\mu_1} \, + \,  \sum_{n=2}^\infty \frac{\mu_{n-1}^{\frac{m+1}{2}}}{\mu_n}  \, 2^{\left(2+\frac{d}{2}(m-1)\right)} \right),
    \]
    where we have used that $\tau_n\leq t$ for all $n\in \N$ and that $Z$ is increasing.
    Choosing $\mu_n = \nu^{n-1}$ for a small enough $\nu>0$, the series on the right hand side is seen to converge. Since $Z(t)\downarrow 0$ as $t\downarrow0$, the result follows. 
    }

    \emph{Ad \ref{thm:point_2}.} Let $x_0\in\T^d\setminus \supp(u_0)$ and $r\in (0,R_0\wedge L)$. We have shown that the parabolic waiting time $T_w^\phi(B_r(x_0))$ is almost-surely positive. We need to show that the absolute waiting time $T_w(B_r(x_0))$ is also positive. To this end, let {$r' \in (r,R_0\wedge L)$}. By \ref{thm:point_1}, we know that $T^\phi_w(B_{r'}(x_0))$ is almost-surely positive. Also almost-surely positive is the stopping time 
    \[
        \rho := \inf \left\{ t\in [0,T]:\, \norm{\phi\inv -\id}_{\infty} > r'-r\right\} \wedge T.
    \]
    By construction, it follows that for any $t\leq \rho$, we have 
    \[
        B_r(x_0) \subset \phi_t(B_{r'}(x_0)).
    \]
    Thus, the time $T_w^\phi(B_{r'}(x_0))\wedge \rho$ is almost-surely positive, and we have 
    \[
        \sup_{0\leq t\leq T_w^\phi(B_{r'}(x_0))\wedge \rho} \, \int_{B_r(x_0)} u^2 \, \leq \sup_{0\leq t\leq T_w^\phi(B_{r'}(x_0))\wedge \rho} \, \int_{\phi_t(B_{r'}(x_0))} u^2  \, = \, 0.
    \]
    Therefore, $T_w(B_r(x_0)) \geq T_w^\phi(B_{r'}(x_0))\wedge \rho$, which completes the proof.

    \emph{Ad \ref{thm:point_3}.} Let $x_0\in \T^d\setminus\supp(u_0)$ be such that $R_0<L$. Suppose that there there exist $r_0>0$ and $f\in \mathscr C_{2r_0,\frac{m+1}{2}}$ such that 
        \[
            S:=\sup_{r\in (0,r_0)} \, \frac{1}{f(r)\, r^{d+\frac{4}{m-1}}} \int_{B_{R_0+r}\setminus B_{R_0}} u_0^2 \,  < +\infty,
        \]
        where we assume here in the following all balls to be centered at $x_0$. 
        
    At this point, we follow \cite{GSU_SPME_Superlinear} again. Like in the proof of \ref{thm:point_1}, it suffices to show that 
    \[
        \lim_{r\downarrow0}\, \PP\left( G(t, L-R_0)>0\right) \, =\, 0.
    \]
    To this end, we let $\xi_0 = L-R_0-r_0$ and $\xi_n = L-R_0 - \frac{r_0}{2^n}$. We define another sequence 
    \[
        \mu_n = \mu \cdot 2^{-\left(d+\frac{4}{m-1}\right)n}\cdot a_n,
    \]
    where $\mu>0$ is arbitrary and the sequence $(a_n)_{n\in\N}$ is such that 
    \[
        \sum_{n=1}^\infty a_{n-1}^{\frac{m-1}{2}}a_n\inv < + \infty, \quad \sum_{n =1}^\infty a_n\inv f(2r_0\cdot 2^{-n}) <+\infty.
    \]
    Recall that the existence of such a sequence is due to the fact that $f\in \mathscr C_{2r_0,\frac{m+1}{2}}$.
    Once again, we have the decomposition 
    \[
        \left\{ G(t,L-R_0) >0 \right\} \subset \bigcup_{n=1}^\infty\left\{ G(\tau_n,\xi_n) > \mu_n \right\},
    \]
    where the stopping times $\tau_n$ are defined in the same way as previously. Applying Markov's inequality, the Stampacchia-type inequality \eqref{prop:stampacchia}, and the definitions of $\mu_n$ and $\tau_n$, and the growth assumption on $u_0$, we find for $n\geq 2$
    \begin{align*}
        \PP(G(\tau_n,\xi_n) >\mu_n) & \leq \, \frac{1}{\mu_n}\e{G(\tau_n,\xi_n)} \\
        &\lesssim_{(m,\gamma,\psi)} \, \frac{1}{\mu_n}\left( \int_{B_{R_0+\frac{r_0}{2^{n-1}}}\setminus B_{R_0}} u_0^2 \, + \, \frac{Z(\tau_n)\, \mu_{n-1}^{\frac{m+1}{2}}}{r_0^{2+\frac{d}{2}(m-1)}}\, 2^{\left(2+\frac{d}{2}(m-1)\right)n}\right) \\
        & = \frac{2^{\left(d+\frac{4}{m-1} \right)n}}{\mu\, a_n}\int_{B_{R_0+\frac{r_0}{2^{n-1}}}\setminus B_{R_0}} u_0^2 \, + \, \mu^{\frac{m-1}{2}}\, Z(t) \, \frac{2^{\left(d+\frac{4}{m-1} \right)\left(\frac{m+1}{2}\right)}}{r_0^{2+\frac{d}{2}(m-1)}}a_{n-1}^{\frac{m-1}{2}}a_n\inv \\
        & \lesssim_{(r_0,S)} \, \frac{a_n\inv}{\mu} f(2r_0\cdot 2^{-n}) \, + \,  \mu^{\frac{m-1}{2}}\, Z(t) \, a_{n-1}^{\frac{m-1}{2}}a_n\inv.
    \end{align*}
    When $n=1$, {we use as in the proof of \ref{thm:point_1} the bound \eqref{eq:energy_est_exp} with $\eta=\mathbf{1}_{\T^d}$}. Summing the estimates over $n$, we obtain 
    \[
        \PP\left( G(t, R_0)>0\right) \, \lesssim_{(m,\gamma,\psi,r_0,S)} \, \frac{1}{\mu}\left( Z(t)\norm{u_0}_{L^2}^{m+1} +  \sum_{n=1}^\infty a_n\inv f(2r_0\cdot 2^{-n}) \right)\, + \, \mu^{\frac{m-1}{2}}\, Z(t)\, \sum_{n=2}^\infty a_{n-1}^{\frac{m-1}{2}}a_n\inv. 
    \]
    Since $f\in \mathscr C_{2r_0,\frac{m+1}{2}}$, the two series on the right-hand side are finite. As $\mu$ is moreover arbitrary, the expression on the right may be made arbitrarily small by first choosing $\mu$ very large and then choosing $t$ appropriately small. This completes the proof.
\end{proof}

{
\begin{proof}[Proof of Corollary \ref{col:lower}]
    This is a consequence of
    \[
        \int_{B_{R_0+ r}\setminus B_r} u_0^2 \, \lesssim \, r^d \, \left(\sup_{B_{R_0+ r}\setminus B_r} u_0\right)^2
    \]
    followed by Theorem \ref{thm:lower}.
\end{proof}
}

\section{Proofs of Theorem \ref{thm:upper} and Corollary \ref{col:upper}}\label{SS:proof_lower}

This section is devoted the proofs of Theorem \ref{thm:upper} and Corollary \ref{col:upper}. To this end, we combine ideas from \cite{ChipotSideris1985} with the random localization technique at the heart of  Theorem \ref{thm:lower}. Let us begin by recalling the setup from \cite{ChipotSideris1985}: For $x_0\in \T^d\setminus \supp(u_0)$ such that $R_0(x_0)<L$ the set 
\begin{equation}
    \label{eq:S1}
    \mathscr S_{x_0} \, = \, \left\{ (y,\epsilon) \in B_{R_0}(x_0) \times [0,L], \quad
             \epsilon^2>\dist{y}{\partial B_{R_0}(x_0)}^2 \geq \frac{d+1}{d+2}\epsilon^2\right\}.
\end{equation}
was already defined in \eqref{eq:S}. For $(y,\epsilon)\in \mathscr S_{x_0}$ we consider moreover the function $\vp_{y,\eps} \in W^{2,\infty}(\T^d)$ given by 
\[
    \vp_{y,\eps}(x) = \left( \epsilon^2-\dist{x}{y}^2\right)_+^2,
\]
present in the necessary condition of Theorem \ref{thm:upper}. The central idea is to derive from the randomly localized mass identity \eqref{eq:loc_mass_general} $\P$-a.s.\ a differential inequality for 
\[
\tint \eta_{y,\epsilon}(t) u(t) ,
\]
where $\eta_{y,\epsilon}$ is the solution to \eqref{eqn_tpe} started from $\vp_{y,\epsilon}$ for suitable $(y,\epsilon)\in  \mathscr S_{x_0}$. The latter is achieved for $0< t\le \sigma_\gamma \wedge T_w^\phi(B_{R_0}(x_0))$ by taking $\gamma$ in \eqref{eq:sigma_gamma} sufficiently small in combination with algebraic properties of the function $\varphi_{y,\epsilon}$. Regarding the latter, we already observe that
 \begin{align}\label{rem2_subs}
            \left\{ x\in \T^d, \, \dist{x}{y}^2 < \frac{d+1}{d+2}\epsilon^2 \right\} \subset B_{R_0}(x_0),
        \end{align}
        due to the assumed inequality $\eps^2< \dist{y}{\partial B_{R_0}(x_0)}^2$. Next, we recall \cite[Lemma 1]{ChipotSideris1985}.

\begin{lemma}
    \label{lemma:CS_argument_usual}
    If $x\in \T^d$ satisfies 
    \begin{equation}
        \label{eq:CS_cond}
        \epsilon^2 > \dist{x}{y}^2 \geq \left( \frac{d+1}{d+2}\right)\epsilon^2,
    \end{equation}
    then for any $m>1$, we have 
    \begin{align}\label{eqn_AAA}
        \Delta \vp_{y,\eps}(x) \, \geq \epsilon^{-4m+2}\vp_{y,\eps}^m(x).
    \end{align}
\end{lemma}

\begin{proof}
    For the reader's convenience, we give the proof, which is a direct computation. First, the restriction \eqref{eq:CS_cond} ensures that $x\in B_\epsilon(y)$ and therefore that $\vp_{y,\eps}$ is smooth at $x$ and both sides of \eqref{eqn_AAA} are well-defined. Secondly,  one can calculate that 
    the following identities hold for all $x\in B_\epsilon(y)$:
    \begin{align}\label{eqn_A}
    \nabla\vp_{y,\eps}(x) = {\color{blue}-}4(\epsilon^2-\dist{x}{y}^2)\, (x-y), \quad D^2\vp_{y,\eps}(x) = 8((x-y)\otimes (x-y)) - 4(\epsilon^2-\dist{x}{y}^2)\,\id. 
    \end{align}
    We remark that in the above and in the following, we identify $x-y \in \T^d$ with its representative in $\R^d$ with minimal length, i.e., with its representative in $B_\epsilon(0)\subset\R^d$. Summing the latter along the diagonal yields in particular that
    \begin{align}\label{eqn_AA}
    \Delta \vp_{y,\eps}(x) = 4\bigl[ (d+2) \, \dist{x}{y}^2-d\epsilon^2\bigr].
    \end{align}
Together with the assumption \eqref{eq:CS_cond}, the latter implies that  
\begin{equation}\label{eqn_1}
	\Delta \vp_{y,\eps}(x) \ge 4 \epsilon^2,
\end{equation}
which together with the trivial $ \vp_{y,\eps} \le \epsilon^4$ yields the claimed \eqref{eqn_AAA}.
\end{proof}

Unlike in the deterministic case, we need actually information on $\eta_{y,\epsilon}$ instead of $\varphi_{y,\epsilon}$. We achieve this through the representation  formula \eqref{eqn_repr_by_stoch_flow}, the sufficiently small choice of $\gamma$ and the following lemma on $\varphi_{y,\epsilon}$.
\begin{lemma}\label{lemma_2}
    If $x\in \T^d$ satisfies \eqref{eq:CS_cond}, then 
    \begin{align}\begin{split}
        \label{eqn_3}
    |\nabla \vp_{y,\eps}(x) | &\lesssim \, \Delta \varphi_{\varepsilon,y}(x), \\
    |D^2\vp_{y,\eps}(x)| &\lesssim \Delta\vp_{y,\eps}(x).
    \end{split}
    \end{align}
\end{lemma}
\begin{proof}
    For any $x\in B_\epsilon(y)$, we find by  \eqref{eqn_A} that 
    \begin{equation}\label{eqn_2}
        |\nabla \vp_{y,\eps} (x) |\le 4\epsilon^3,
    \end{equation}
     and regarding the Hessian, we compute 
    \begin{align}
        & \mathrm{tr} \bigl( D^2\vp_{y,\eps}(x)^\top D^2\vp_{y,\eps}(x) \bigr)
        \\&\quad  = 64 \dist{x}{y}^4
 - 64 \dist{x}{y}^2 (\epsilon^2-\dist{x}{y}^2)  + 16 d (\epsilon^2-\dist{x}{y}^2)^2 
 \\&\quad  = 64 \dist{x}{y}^2 (2\dist{x}{y}^2 - \epsilon^2)  + 16 d (\epsilon^2-\dist{x}{y}^2)^2 .
    \end{align}
    This implies then
    \[
       \abs{D^2\vp_{y,\eps}(x)}  = 4\sqrt{4\dist{x}{y}^2 (2\dist{x}{y}^2 - \epsilon^2)  +  d (\epsilon^2-\dist{x}{y}^2)^2 } 
    \]
    and using once more $x\in B_\epsilon(y)$ we may bound the above by 
    \[
    4\sqrt{4\epsilon^4  +  d \epsilon^4 } = 4 \sqrt{4+d} \epsilon^2
    .
    \] In light of \eqref{eqn_1} and \eqref{eqn_2} we deduce  \eqref{eqn_3}.
\end{proof}

We are now ready to prove Theorem \ref{thm:upper} and Corollary \ref{col:upper} on necessary conditions for parabolic waiting time phenomena of \eqref{Eq_SPME_intro}. 
\begin{proof}[Proof of Theorem \ref{thm:upper}]
    Let $x_0\in \T^d\setminus \supp(u_0)$ such that $R_0<L$. Recall that we assumed
    \begin{equation}
        \label{eq:proof_assumption_blowup}
        \inf_{\mathscr S_{x_0}} \,\left\{ \left(\frac{1}{\epsilon^{4+d+\frac{2}{m-1}}} \tint \vp_{y,\epsilon}\, u_0 \right)^{1-m} \right\} \,  = \, 0
    \end{equation}
    in \eqref{eqn_12},
    and thereby, we can take for any $\delta>0$, $(y,\epsilon)\in\mathscr S_{x_0}$ such that 
    \begin{align}
        \left(\frac{1}{\epsilon^{4+d+\frac{2}{m-1}}} \tint \vp_{y,\epsilon}\, u_0 \right)^{1-m} <\delta.\label{eq:yadayada}
    \end{align}
    For this choice of $(y,\epsilon)$ we let $\eta_{y,\epsilon}$ be the solution to \eqref{eqn_tpe} started from $\vp_{y,\eps}$ provided by Lemma \ref{lemma:representation_formula}. We recall once more the  representation formula
    \[
    \eta_{y,\epsilon}(t,x) = \vp_{y,\eps}(\phi_t\inv(x)), \qquad (t,x) \in  [0,T] \times\T^d,
    \]
    $\P$-a.s., stated in \eqref{eqn_repr_by_stoch_flow}.
    Lemma \ref{lemma:loc_mass} yields then that, $\P$-a.s.\ for all $t\in [0,T]$:
    \[
        \tint \eta_{y,\epsilon}(t) u(t) = \tint \vp_{y,\eps}u_0 +  \int_0^t\tint  u^m \Delta\eta_{y,\epsilon} . 
    \]
    Also $\P$-a.s., for  $0<t \le T_w^\phi(B_{R_0}(x_0))$, we use the fact that $\phi$ is a flow of homemorphisms to see
    \begin{align*}
        \supp(u(t)\Delta\eta_{y,\epsilon}(s) ) & = \phi_t(B_\epsilon(y)) \cap \supp(u(t))
        \\& \subset\phi_t(B_\epsilon(y)) \cap \phi_t((B_{R_0}(x_0))^c)
        \\& = \phi_t(B_\epsilon(y)\setminus B_{R_0}(x_0)).
    \end{align*}
    This allows us to verify that, $\P$-a.s., for each $0<t\le T_w^\phi(B_{R_0}(x_0))$ and $x\in \supp(u(t)\Delta\eta_{y,\epsilon}(t) )$, $\phi_t\inv(x)$ satisfies the condition \eqref{eq:CS_cond}. Indeed, the above entails 
    \[
        \phi_t\inv(x) \in B_\epsilon(y)\setminus B_{R_0(x_0)}(x_0).
    \]
    In particular, we have $\dist{\phi_t\inv(x)}{y}^2<\epsilon^2$, which by \eqref{rem2_subs} tells us that 
    \[
        \dist{\phi_t\inv(x)}{y}^2 \, \ge \, \left( \frac{d+1}{d+2} \right)\epsilon^2.
    \]    
    Hence \eqref{eq:CS_cond} is fulfilled, and applications of Lemmas \ref{lemma:CS_argument_usual} and \ref{lemma_2} yield 
    \begin{equation}\label{eq:proof_upper_est}\begin{split}
        (\Delta \vp_{y,\eps}\circ\phi_t \inv)(x) \, &\geq \, \epsilon^{-4m+2}(\vp_{y,\eps}^m \circ \phi_t\inv)(x), \\
        |(\nabla \vp_{y,\eps}\circ\phi_t\inv)(x) | &\lesssim \, (\Delta \varphi_{\varepsilon,y}\circ\phi_t\inv)(x), \\
        |(D^2\vp_{y,\eps}\circ\phi_t\inv)(x)| &\lesssim(\Delta\vp_{y,\eps}\circ\phi_t\inv)(x),
        \end{split}
    \end{equation}$\P$-a.s., for each  $0<t \le T_w^\phi(B_{R_0}(x_0))$, and $x\in \supp(u(s)\Delta\eta_{y,\epsilon}(s) )$.

    Recall the possibly small, but $\P$-a.s.\ positive stopping time $\sigma_\gamma$  introduced in \eqref{eq:sigma_gamma}. We claim that for a small enough choice of $\gamma>0$, the following holds $\P$-a.s.\ for $0<t\leq \sigma_\gamma\wedge T_w^\phi(B_{R_0}(x_0))$: 
    \begin{align}\label{eqn_Albert}
        \int_0^t\tint  u^m \Delta\eta_{y,\epsilon} \, \geq \, \frac{1}{2}\int_0^t\tint  u^m\,  (\Delta\vp_{y,\eps}\circ\phi_s\inv) .
    \end{align}
    To see this, we estimate $\Delta\eta_{y,\epsilon}$ at $x\in \supp(u(t)\Delta \eta_{y,\epsilon}(t))$ as follows, using \eqref{eqn_formula}:
    \begin{align*}
        \Delta\eta_{y,\epsilon} 
        & = \mathrm{tr}\left( (D\phi_t\inv)^\top(D^2\vp_{y,\eps}\circ \phi_t\inv) (D\phi\inv_t)\right) \, + \, (\nabla\vp_{y,\eps}\circ \phi_t\inv) \cdot \Delta\phi_t\inv \\
        & = \Delta\vp_{y,\eps}\circ\phi_s\inv + \mathrm{tr}\left( (D\phi_t\inv - \id)^\top (D^2\vp_{y,\eps}\circ \phi_t\inv) \id \right) \\
        &\qquad + \mathrm{tr}\left( (D\phi_t\inv)^\top (D^2\vp_{y,\eps}\circ \phi_t\inv) (D\phi\inv_t- \id)\right)\, + \, (\nabla\vp_{y,\eps}\circ \phi_t\inv) \cdot \Delta\phi_t\inv \\
        & \geq \Delta\vp_{y,\eps}\circ\phi_s\inv - \abs{D\phi_t\inv - \id}\abs{D^2\vp_{y,\eps}\circ \phi_t\inv}\bigl(\abs{\id}+\abs{D\phi\inv_t} \bigr) - \abs{\nabla\vp_{y,\eps}\circ \phi_t\inv}\abs{\Delta\phi_t\inv} \\
        & \geq \Delta\vp_{y,\eps}\circ\phi_s\inv \bigl( 1 - C \abs{D\phi_t\inv - \id}\abs{D^2\vp_{y,\eps}\circ \phi_t\inv}\bigl(\abs{\id}+\abs{D\phi\inv_t} \bigr) -C \abs{\Delta\phi_t\inv} \bigr),
    \end{align*}
    where the last inequality above follows from \eqref{eq:proof_upper_est}.
    As a sufficiently small choice of $\gamma$ diminishes the factors of the implicit constant $C$, \eqref{eqn_Albert} follows. Using also the first line of \eqref{eq:proof_upper_est}, we deduce that 
    \[
        \tint \eta_{y,\epsilon}(t) u(t) \geq \tint \vp_{y,\eps}u_0 \, + \,  \frac{\epsilon^{-4m+2}}{2}\int_0^t\tint \eta_{y,\epsilon}^m \, u^m,
    \]
    $\P$-a.s., for all $0<t\leq \sigma_\gamma\wedge T_w^\phi(B_{R_0}(x_0))$. 
    Applying Hölder's inequality and noticing that $\mathcal L^d(\phi_s(B_\epsilon(y)\setminus B_{R_0(x_0)}(x_0))) \lesssim_{\gamma} \epsilon^d$, we obtain
    \[
        \tint \eta_{y,\epsilon}(t) u(t) \geq \tint \vp_{y,\eps}u_0 \, + \,  \frac{\epsilon^{-4m+2+d(1-m)}}{2 \, c(d,\gamma)}\int_0^t\left(\tint \eta_{y,\epsilon} \, u\right)^m,
    \]
    i.e., the left-hand side of the above satisfies $\P$-a.s.\ a differential inequality. This allows for a comparison to the
     solution to the ordinary differential equation 
    \[
        y'(t) = By(t)^m, \qquad B = \frac{\epsilon^{-4m+2+d(1-m)}}{2 \, C_\gamma },
    \]
     given by 
    \[
        y(t) = \left( y(0)^{1-m} - B(m-1) \, t \right)^{\frac{1}{1-m}}, \qquad  t < \frac{y(0)^{1-m}}{(m-1)B}.
    \]
    Using that at the same time $\int \eta_{y,\epsilon} u \leq \epsilon^4\norm{u}_{L^1}$, this entails by conservation of mass that 
    \[
        T_w^\phi(B_{R_0(x_0)}(x_0)) \wedge \sigma_\gamma \, \lesssim_{(m,\gamma)} \, \left(\epsilon^{-d-\frac{2}{m-1}-4}\tint \vp_{y,\eps}u_0\right)^{1-m}<\delta,
    \]
    by \eqref{eq:yadayada}.
    Since $\delta>0$ was arbitrary, we obtain  that $T_w^\phi(B_{R_0(x_0)}(x_0)) \wedge \sigma_\gamma = 0$, $\P$-almost surely. Since $\sigma_\gamma$ is $\P$-a.s.\ positive, this yields the desired 
    \begin{align}
        \P( T_w^\phi(B_{R_0(x_0)}(x_0))  =0 )=1.
    \end{align}
\end{proof}

\begin{proof}[Proof of Corollary \ref{col:upper}]
    This proof follows along the lines of \cite[p. 426]{ChipotSideris1985}. Suppose that $x_0\in \T^d\setminus \supp(u_0)$ such that $R_0<L$ and $\tilde x \in \partial B_{R_0}(x_0)\cap \supp(u_0)$ such that \eqref{eqn_bshfdsf} holds, i.e., 
    \[
        \limsup_{\epsilon\downarrow 0} \, \inf_{\mathcal C(\tilde x)\cap  B_\epsilon(\tilde x)}\, \frac{u_0(x)}{\dist{x}{\tilde x}^{\frac{2}{m-1}}} = +\infty,
    \]
    for a truncated cone $\mathcal C(\tilde x)$ centered at $\tilde x$.
    In particular, for each $n\in \N$, we may find $\epsilon_n\in(0,L/n)$ with
    \[
        \inf_{\mathcal C(\tilde x)\cap  B_{\epsilon_n}(\tilde x)}\, \frac{u_0(x)}{\dist{x}{\tilde x}^{\frac{2}{m-1}}} > n.
    \]
    Therefore, for all $n\in \N$, we have 
    \[
        \frac{1}{\epsilon_n^{4+d+\frac{2}{m-1}}} \tint \left( \epsilon_n^2 -\dist{x}{y_n}^2\right)_+^2\, u_0(x) \, \d x >\frac{n}{\epsilon_n^{4+d+\frac{2}{m-1}}} \int_{ \mathcal C(\tilde x)\cap  B_{\epsilon_n}(\tilde x)} \left( \epsilon_n^2 -\dist{x}{y_n}^2\right)_+^2\, \dist{x}{\tilde x}^{\frac{2}{m-1}} \, \d x,
    \]
    for any $(y_n,\epsilon_n)\in \mathscr S_{x_0}$. We may choose $y_n$ belonging to the line-segment that connects $x_0$ to $\tilde x$. Note then that $\eps_n\sqrt{\frac{d+1}{d+2}} \leq \dist{\tilde x}{y_n}< \eps_n$ by the definition \eqref{eq:S1} of $\mathscr S_{x_0}$, and we choose the $y_n$ such that the former inequality becomes an identity. Using the triangle inequality, for any $\Sigma>0$ and for all $x\in B_{\eps_n\Sigma^2}(\tilde x)$, we have {therefore}
    \begin{align*}
        \dist{x}{y_n}^2 &\leq \left( \dist{x}{\tilde x} + \dist{\tilde x}{y_n} \right)^2
        \\& =  \dist{x}{\tilde x}^2 + 2 \dist{x}{\tilde x}\dist{\tilde x}{y_n} + \dist{\tilde x}{y_n}^2
        \\&\le \dist{x}{\tilde x}^2 + \left(2\Sigma^2\sqrt{\frac{d+1}{d+2}}+\frac{d+1}{d+2}\right)\eps_n^2.
    \end{align*}
    By choosing $\Sigma$ small enough (and independently of $n$) in a way that 
    \[
        1-\frac{d+1}{d+2}-2\Sigma^2 \sqrt{\frac{d+1}{d+2}} > \Sigma^2,
    \]
    we have for all $x\in B_{\eps_n\Sigma^2}(\tilde x)$
    \begin{align*}
                \epsilon_n^2 -\dist{x}{y_n}^2 \geq & \left(1-\frac{d+1}{d+2}-2\Sigma^2 \sqrt{\frac{d+1}{d+2}} \right)\eps_n^2 - \dist{x}{\tilde x}^2  \\
                \geq & \, \eps_n^2\Sigma^2 - \dist{x}{\tilde x}^2.
    \end{align*}
    Taking additionally $\Sigma<1$, it follows
    \[
         \int_{ \mathcal C(\tilde x)\cap  B_{\epsilon_n}(\tilde x)} \left( \epsilon_n^2 -\dist{x}{y_n}^2\right)_+^2\, \dist{x}{\tilde x}^{\frac{2}{m-1}} \, \d x \geq  \int_{ \mathcal C(\tilde x)\cap B_{\eps_n\Sigma^2}(\tilde x)} \left( \epsilon_n^2\Sigma^2 -\dist{x}{\tilde x}^2\right)_+^2\, \dist{x}{\tilde x}^{\frac{2}{m-1}} \, \d x.
    \]
    The latter integral may be computed explicitly,  using for instance spherical coordinates (up to choosing $\Sigma$ smaller but still independent of $n$), to identify
    \[
        \mathcal C(\tilde x)\cap B_{\eps_n\Sigma^2}(\tilde x) \eqsim \left\{ (r,\phi_1,\dots,\phi_{d-1})\in [0,\eps_n\Sigma^2]\times [0,\theta] \times [0,\pi]^{d-3}\times [0,2\pi] \right\},
    \]
    where $2\theta$ is the aperture of $\mathcal C(\tilde x)$. We moreover  restrict ourselves to the more delicate situation that $d\ge 3$.  The volume element is then given by 
    \[
        \d x = r^{d-1} \sin^{d-2}(\phi_1)\sin^{d-3}(\phi_2)\cdots \sin(\phi_{d-2})\,  \d r \, \d \phi_1 \, \d\phi_2  \cdots \d \phi_{d-1}.
    \]
    Therefore, using again that $\Sigma<1$, we may calculate
    \begin{align*}
        &\int_{ \mathcal C(\tilde x)\cap B_{\eps_n\Sigma^2}(\tilde x)} \left( \epsilon_n^2\Sigma^2 -\dist{x}{\tilde x}^2\right)_+^2\, \dist{x}{\tilde x}^{\frac{2}{m-1}} \, \d x 
        \\ &\quad = C(d) \int_0^{\epsilon_n\Sigma^2} \left( \epsilon_n^2\Sigma^2 - r^2 \right)^2 r^{d-1+\frac{2}{m-1}}\, \d r\int_0^\theta \sin^{d-2}(\phi_1)\, \d \phi_1 
        \\&\quad  = C(d,\theta,m,\Sigma) \,\epsilon_n^{4+d+\frac{2}{m-1}}.
    \end{align*}
    Putting all of this together, there exists a constant $C(d,m,\theta,\Sigma)\in (0,\infty)$ such that for each $n\in \N$, there exists $(y_n,\eps_n)\in \mathscr S_{x_0}$ with
    \[
        \frac{1}{\epsilon_n^{4+d+\frac{2}{m-1}}} \tint \left( \epsilon_n^2 -\dist{x}{y_n}^2\right)_+^2\, u_0(x) \, \d x \ge C(d,m,\theta,\Sigma) \, n .
    \]
    We infer therefore that 
    \[
            \inf_{\mathscr S_{x_0}} \,\left\{ \left(\frac{1}{\epsilon^{4+d+\frac{2}{m-1}}} \tint \left( \epsilon^2 -\dist{x}{y}^2\right)_+^2\, u_0(x) \, \d x\right)^{1-m} \right\} \,  = \, 0,
    \]
    and we may apply Theorem \ref{thm:upper} to conclude.
\end{proof}

\appendix 

\section{Interpolation of probabilistic convergences}\label{app:interp}

To obtain the additional mode of convergence  \eqref{eqn_convergence_by_boundedness_LinftyL2} of the approximations of the kinetic solution to \eqref{Eq_SPME_intro}, we employ the following abstract result. For a similar observation regarding tightness of random variables, we refer to \cite[Lemma D.1]{NSF-paper}.
\begin{lemma}\label{lemma_app_A}
	Let $X_1\hookrightarrow  X_\delta \hookrightarrow X_0$ be Banach spaces, such that $X_0 \ni  u\mapsto \|u\|_{X_1} \in \R \cup\{\infty\}$ is lower semicontinuous and  the interpolation inequality
	\begin{equation}\label{eqn_app}
		\|  u \|_{X_\delta}\,\lesssim\, \|  u \|_{X_0}^{1-\delta} \|  u \|_{X_1}^\delta,
	\end{equation}
	holds for some $\delta \in (0,1)$. Let $(u_\epsilon)_{\epsilon\in (0,1)}$ and $u$ be $X_\delta$-valued random variables such that  $u_\epsilon \to u $  in $ X_0$, in probability, and $u_\epsilon$ has uniform tail estimates in $X_1$, i.e., 
	\[\biggl(\sup_{\epsilon\in (0,1)} \P ( \|u_\epsilon \|_{X_1} > N) \biggr)\,\to\, 0,
	\]
	as $N\to\infty$. Then $u_\epsilon \to  u$ also in $X_\delta$, in probability.
\end{lemma}
\begin{proof}
	First, we observe that since there exists a subsequence, for which $u_\epsilon \to u$ in $X_0$ also $\P$-a.s., it follows that 
	\[
	\liminf_{\epsilon\to 0} \| u_\epsilon \|_{X_1} \ge  \| u \|_{X_1}, \qquad \P\text{-a.s.,}	\]
	along this subsequence. Then, using that the function $\mathbf{1}_{ (N,\infty)}$ is lower semicontinuous and non-decreasing, we find that 
	\begin{align}
		\liminf_{\epsilon\to 0 } \mathbf{1}_{\{
	\| u_\epsilon \|_{X_1} >N \}} \ge 	 \mathbf{1}_{\{\liminf_{\epsilon\to 0 }
	\| u_\epsilon \|_{X_1} >N \}} \ge 	 \mathbf{1}_{\{
	\| u \|_{X_1} >N \}} ,
	\end{align}
	and thus by Fatou's lemma
	\begin{align}\liminf_{\epsilon\to 0 } \P(
		\| u_\epsilon \|_{X_1} >N  ) \ge \P (	\| u \|_{X_1} >N ) ,
	\end{align}
	showing that $u$ obeys the same tail estimates as the sequence $u_\epsilon$. 
	
	Then, returning to the original sequence, we may use \eqref{eqn_app} to the effect that
	\begin{align}
		\| u_\epsilon - u\|_{X_\delta}  \le C \| u_\epsilon - u \|_{X_0}^{1-\delta} ( \| u_\epsilon\|_{X_0}^\delta + \| u\|_{X_0}^\delta) .  
	\end{align}
	Thus, for any $\gamma>0$ and $N\in \N$, we may estimate
	\begin{align}
		\P(\| u_\epsilon - u\|_{X_\delta} >\gamma) &\le \P\bigl(\| u_\epsilon - u \|_{X_0}^{1-\delta}> \gamma/(2CN) \bigr) + \P\bigl(\| u_\epsilon\|_{X_0}^\delta >N \bigr)   +  \P\bigl(  \| u\|_{X_0}^\delta >N\bigr)
		\\& =\P\bigl(\| u_\epsilon - u \|_{X_0}> (\gamma/(2CN))^{1/(1-\delta)} \bigr) + \P\bigl(\| u_\epsilon\|_{X_0} >N^{1/\delta} \bigr)   +  \P\bigl(  \| u\|_{X_0} >N^{1/\delta}\bigr) .
	\end{align}
	Since the right-hand side may be diminished by first taking $N$ sufficiently large and then $\varepsilon$ sufficiently small, the claim follows. 
\end{proof}

\section{An infinite-dimensional It\^o formula for products} \label{app:ito}

This appendix is dedicated to the proof of a mixed It\^o product and composition rule for solutions to  SPDEs with common noise. It is tailored towards our application and can be seen as a variant of \cite[Proposition A.1]{DHV_16} for test functions which vary randomly over time. To this end, we consider in the following adapted processes 
\begin{equation}\label{eqn_67}
    u\in  C([0,T];L^2(\T^d)) \cap L^2(0,T; H^1(\T^d))  \qquad \text{and} \qquad  \eta \in  C ([0,T]; C ( \T^d))\cap  L^2 (0,T; C^1 ( \T^d)) ,
\end{equation}
that satisfy  $\P$-a.s., for all $t\in [0,T]$,
\begin{equation}
    \label{eq:ito_prod_rule1}
    \begin{split}
        \d u = &F \, \d t + \diver(G) \, \d t + \sum_{k=1}^{\infty} H_k\cdot \d \beta_k,  \\
        \d \eta =& A  \, \d t + \diver(B) \, \d t   + \sum_{k=1}^{\infty} E_k  \cdot \d \beta_k.
    \end{split}
\end{equation}
The appearing deterministic and stochastic integrands are assumed to be progressively measurable and satisfy
\begin{equation}
    \label{eq:ass_ito_prod_regularity}
    \begin{split}
        F &\in L^{{1}}(0,T;L^2(\T^d)) ,\\
        G &\in L^2(0,T;L^2(\T^d;\R^d)), \\
        H &\in L^2(0,T;L^2(\T^d;  \ell^2 (\R^d))), \\
        A & \in L^1 ( 0,T ; C(\T^d)) ,\\
        B & \in L^{{2}} ( 0,T ; C(\T^d;\R^d))    ,\\
        E & \in L^2(0,T; C(\T^d;\ell^2 (\R^d))),
    \end{split}
\end{equation}
$\P$-almost surely. Accordingly, the identities in \eqref{eq:ito_prod_rule1} are understood  
in $H^{-1}(\T^d)$.
\begin{prop}
    \label{prop:Ito_product_rule}
    Let $f\in C^2(\R)$ with bounded second derivative. 
    If $u$ and $\eta$ satisfy \eqref{eq:ito_prod_rule1} under the above assumptions \eqref{eqn_67} and \eqref{eq:ass_ito_prod_regularity}, then the following holds $\P$-a.s.\ for every $t\in [0,T]$: 
    \begin{equation}
        \label{eq:ito_prod_rule_main}
        \begin{split}&
            \int_{\T^d}{\eta(t)}{f(u(t) ) }  =  \int_{\T^d}{\eta_0}{f(u_0)} 
             + \int_0^t \int_{\T^d}{\eta f'(u)}{F}  - \int_0^t  \int_{\T^d}{\nabla(\eta f'(u ))}\cdot {G} 
            + \int_0^t \int_{\T^d}{f(u)}{A}  \\&\qquad   - \int_0^t  \int_{\T^d}{f'(u)\nabla u}\cdot {B}  +\frac{1}{2}\sum_{k=1}^{\infty} \int_0^t  \int_{\T^d}{\eta f''(u) }{\abs{H_k}^2}+ \sum_{k=1}^{\infty}\int_0^t \int_{\T^d}{f'(u) H_k}\cdot{E_k} \\
            & \qquad + \sum_{k=1}^{\infty} \int_0^t  \int_{\T^d}{\eta f'(u ) }{H_k \cdot \d \beta_k }  + \sum_{k=1}^{\infty} \int_0^t  \int_{\T^d}{f(u)}{E_k\cdot \d \beta_k }.
        \end{split}
    \end{equation}
\end{prop}

\begin{proof}The proof of the above  amounts to a regularization procedure, as performed for instance in \cite[Proposition A.1]{DHV_16} for the case that $A$, $B$ and $E$ all vanish. As  there,  we let $\rho^\delta$ be a smooth approximation of the identity on $\T^d$ and denote the convolution $g*\rho^\delta$ of a distribution $g$ on $\T^d$ in the following  by $g^\delta$. Testing then the  equation for $u$ in \eqref{eq:ito_prod_rule1} with $\rho^\delta(x-\cdot)$, we find that $u^\delta$ satisfies
    \[
        \d  u^\delta  = F^\delta \, \d t + \diver(G^\delta )\, \d t + \sum_{k=1}^{\infty} H_k^\delta \cdot \d \beta_k,
    \]
    for each $x\in \T^d$. Similarly, we obtain that 
    \begin{align}
        \d \eta^\delta  =& A^\delta  \, \d t + \diver(B^\delta )\, \d t + \sum_{k=1}^{\infty} E_k^\delta  \cdot \d \beta_k.
    \end{align}

    Whence, the scalar It\^o formula yields  
    \begin{align}
        \d f( u^\delta)  = f'(u^\delta) F^\delta \, \d t + f'(u^\delta) \diver(G^\delta )\, \d t +  \frac{1}{2}\sum_{k=1}^\infty f''(u^\delta) |H_k|^2\,\d t +\sum_{k=1}^{\infty} f'(u^\delta) H_k^\delta \cdot \d \beta_k,
    \end{align}
    and  a subsequent application of It\^o's product rule that
    \begin{align} 
         \d \eta^\delta f(u^\delta)  = & \eta^\delta f'(u^\delta) F^\delta \, \d t + \eta^\delta f'(u^\delta) \diver(G^\delta )\, \d t + f(u^\delta )A^\delta \,\d t   + f(u^\delta ) \diver (B^\delta)   \,\d t \\&  +  \frac{1}{2} \sum_{k=1}^\infty \eta^\delta f''(u^\delta) |H_k^\delta |^2\,\d t + \sum_{k=1}^{\infty}f'(u^\delta) H_k^\delta \cdot E_k^\delta \, \d t +\sum_{k=1}^{\infty} \eta^\delta  f'(u^\delta) H_k^\delta \cdot \d \beta_k + f(u^\delta) E_k^\delta \cdot \d \beta_k .
    \end{align}
    Since both sides of the above also make sense as a stochastic process which is continuous in $C(\T^d)$, this pointwise identity allows to identify the function valued processes with each other. Integration in space, which is a continuous linear operation  on said space, yields then    
    \begin{equation}\label{eqn_2384}
        \begin{split}&
            \int_{\T^d} {\eta^\delta(t)}{f(u^\delta(t)) }  = \int_{\T^d}{\eta_0^\delta}{f(u_0^\delta )} 
             + \int_0^t\int_{\T^d}{\eta^\delta f'(u^\delta)}{F^\delta} - \int_0^t \int_{\T^d}{\nabla(\eta^\delta f'  (u^\delta ))}\cdot {G^\delta} 
             + \int_0^t\int_{\T^d}{f(u^\delta)}{A^\delta} \\&\qquad   - \int_0^t\int_{\T^d}{f'(u^\delta) \nabla u^\delta }\cdot {B^\delta} +\frac{1}{2}\sum_{k=1}^{\infty} \int_0^t \int_{\T^d}{\eta^\delta f''(u^\delta) }{\abs{H_k^\delta}^2} + \sum_{k=1}^{\infty}\int_0^t\int_{\T^d}{f'(u^\delta) H_k^\delta}\cdot {E_k^\delta} \\
            & \qquad + \sum_{k=1}^{\infty} \int_0^t \int_{\T^d}{\eta^\delta f'(u^\delta ) }{H_k^\delta \cdot \d \beta_k }  + \sum_{k=1}^{\infty} \int_0^t \int_{\T^d}{f(u^\delta)}{E_k^\delta\cdot \d \beta_k },
        \end{split}
    \end{equation}
    where we integrated by parts in the third and fifth term in the right-hand side. Let us point out that this is a convolved version of the desired \eqref{eq:ito_prod_rule_main} and it remains to make an argument that the convolution may be taken out: Firstly, since both sides of \eqref{eq:ito_prod_rule_main} define continuous in time processes, we may argue for fixed $t\in [0,T]$. Secondly, we record that $\P$-a.s., the mollifications converge in the spaces specified in \eqref{eqn_67} and \eqref{eq:ass_ito_prod_regularity} by standard vector-valued analysis,
    cf.\ \cite[Section 1.2]{Analysis_1}. Thereby, we can use for the second term on the right-hand side for instance,  that $L^2(\T^d) \ni u \to f'(u) \in L^2(\T^d)$ is Lipschitz continuous guaranteeing the convergence $f'(u^\delta) \to f'(u) $ in $C([0,T]; L^2(\T^d))$, $\P$-a.s., allowing to take the limit in the integrand. Regarding the third term, we  make use of the product rule to conclude
    \begin{align}
        \nabla(\eta^\delta f'  (u^\delta )) = \eta^\delta \nabla f'(u^\delta) + f'(u^\delta) \nabla \eta^\delta \to 
        \eta \nabla f'(u ) + f'(u ) \nabla \eta = \nabla ( \eta f'(u) ) 
    \end{align}
    in $L^2(0,T; L^2(\T^d;\R^d))$, $\P$-a.s., where emloyed that $f'(u^\delta) \to f'(u)$ also in $L^2(0,T;H^1(\T^d))$ due to 
    \[
    \nabla f'(u^\delta) - \nabla f'(u) = f''(u^\delta) (\nabla u^\delta - \nabla u)  + (f''(u^\delta) -f''(u)) \nabla u
    \] 
    and dominated convergence. Regarding the fourth term, we use that  
    $L^2(\T^d) \ni u \to f(u) \in L^1(\T^d)$ is continuous and thus $\P$-a.s.\ uniformly continuous on the totally bounded set $\{u_t(\omega), u^{\delta_l}_t(\omega) | t\in [0,T] ; l\in \N \}$ for any $\delta_l \to 0$. This implies that $\P$-a.s.\ $f(u^{\delta_l} ) \to f(u) $ in $C([0,T]; L^1(\T^d))$ and since $\delta_l$ was arbitrary, this also holds for $f(u^{\delta} )$. Together with the convergence of $A_\delta \to A$ this yields that 
    \begin{align}
        \int_0^t\int_{\T^d}{f(u^\delta)}{A^\delta} \to \int_0^t\int_{\T^d}{f(u)}{A}  ,
    \end{align}
    $\P$-almost surely. Variants of these arguments allow to deal also with the remaining deterministic integrals on the right-hand side of \eqref{eqn_2384}.
    
    Let us comment on how to handle the local martingales in \eqref{eqn_2384}: For this we  use that the convergence  $\langle M^\delta \rangle_t \to 0$ in probability implies $M_t^\delta \to 0$ in probability, for local martingales $M^\delta$. Since the quadratic variations of 
    \begin{align}&
        \sum_{k=1}^{\infty} \int_0^t \int_{\T^d}{\eta^\delta f'(u^\delta ) }{H_k^\delta \cdot \d \beta_k }  - \int_0^t \int_{\T^d}{\eta f'(u ) }{H_k \cdot \d \beta_k }  \\\text{and}\qquad  &  \sum_{k=1}^{\infty} \int_0^t \int_{\T^d}{f(u^\delta)}{E_k^\delta\cdot \d \beta_k } - \int_0^t \int_{\T^d}{f(u)}{E_k \cdot \d \beta_k },
    \end{align}
    are 
    \begin{align}
    \sum_{k=1}^{\infty} \int_0^t  \biggl|\int_{\T^d}
    \eta^\delta f'(u^\delta )  H_k^\delta - \eta f'(u ) H_k
    \, \d x \biggr|^2 \d s
    \qquad 
        \text{and} \qquad 
         \sum_{k=1}^{\infty} \int_0^t \biggl|  f(u^\delta) E_k^\delta - f(u) E_k  \,\d x \biggr|^2 \d s,
    \end{align}
    respectively, the aforementioned methods lead the desired convergences.
\end{proof}

\section{Periodic functions with compact effective support}
Here, we introduce the notion of a  periodic function having compact effective support and resulting consequences. We start with the definition.
\begin{defn}\label{defi_CES}
    We say that a measurable function  $\vp\colon  \T^d\to \R$ has \emph{compact effective support}, if there exists a compact set $C\subset \R^d$, such that \begin{align}\label{comp_eff_supp}\pi\inv(\supp(\vp)) = \bigcup_{k\in \Z^d} 2Lk + C %, \quad C \text{ compact,}
    \end{align}
    obeys $(2Lj+C) \cap (2Lk+C) = \emptyset$ for all $j\neq k$.
\end{defn}

The point of the above definition is to identify functions on $\T^d$, for which properties of $C_c^\infty(\R^d)$-functions may transferred. For the purpose of the current manuscript, we show that  the homogeneous Gagliardo--Nirenberg inequality holds for periodic functions with compact effective support. Notice that the resulting \eqref{eq-GN-1} does not hold without any restrictions, as can be seen by considering $u = \mathbf{1}_{\T^d}$. 
For its proof, we first recall the homogeneous Gagliardo--Nirenberg inequality on the whole space allowing also for integrability exponents below $1$, as stated for instance in \cite[Proposition A.1]{Passo_Giacomelli_Shishkov}. 

\begin{lemma} \label{app-L-1}
Let $ 1 \le r \le \infty$, $0< q <p< \infty$
such that
\begin{align} \label{GN-assumption}
\frac 1r - \frac md < \frac 1p \,.
\end{align}
Then, for any $u \in L^q (\R^d)$
satisfying $ D^m u \in L^r (\R^d)$, the following inequality holds:
\begin{equation}
\| u \|_{L^p(\R^d)} \lesssim_{(m,p,q,r)}  \|D^m u\|^\theta_{L^r(\mathcal O)} \| u \|^{1-\theta}_{L^q(\mathcal O)},
\label{eq-GN}
\end{equation}
where $\theta = \frac{\frac 1q - \frac 1p}{\frac 1q + \frac m d- \frac
1r}$.
\end{lemma}

\begin{col}\label{app-L-2}
Let $p,q,r$ and $\theta$ as in Lemma 
\ref{app-L-1}, then it holds 
\begin{equation}
\| u \|_{L^p(\T^d)} \lesssim_{(m,p,q,r)}  \|D^m u\|^\theta_{L^r(\T^d)} \| u \|^{1-\theta}_{L^q(\T^d)}, 
\label{eq-GN-1}
\end{equation}
for any $u \in L^q (\T^d)$ which has compact effective support and
satisfies $ D^m u \in L^r (\T^d)$.
\end{col}

\begin{proof}
    Without loss of generality, we suppose that $L=1$ so that $\T^d=\R^d/\Z^d$. Let $u\in L^q (\T^d)$ satisfy $ D^m u \in L^r (\T^d)$ and have compact effective support. The latter means that there exists a compact set $C \subset \R^d$ such that 
    \begin{align}\label{eqn_GG}
        \pi\inv \left(\supp(u)\right) = \bigcup_{k\in \Z^d} k+ C, \qquad(j+C)\cap (k+C)  = \emptyset, \quad j\ne k,
    \end{align}
    where $\pi:\R^d \ra \R^d/\Z^d$ is the quotient map. Let $\tilde u:\R^d\ra \R$ be the periodic extension of $u$ to $\R^d$, i.e., $\tilde u(x) = (u\circ \pi)(x)$. 
    
    We use topological properties of the above decomposition of the support of $\tilde u$ as follows: First, since $C$ is compact, the union of closed sets 
    \begin{align}
        A = \bigcup_{k\in \Z^d\setminus \{0\}} k + C
    \end{align}
    is locally finite and therefore closed again. In particular, the distance function $d(\cdot, A)  =\inf_{x\in A} d(\cdot ,x)$ attains a positive  minimum $d(C,A)$ on $C$. Thus, there exists an open $U\supset C$ such that $d(x,C) \le  d(C,A)/3$ for all $x\in \overline{U}$. It follows that also $d(y,A)\le d(C,A)/3$ for all 
    \[
    y\in \bigcup_{k\in \Z^d\setminus\{0\}} k+ \overline{U}.
    \]
    In particular, 
    \begin{equation}\label{eqn_GGG}
    (j+\overline{U})\cap (k+\overline{U}) = \emptyset ,
    \end{equation}
    for each $k\in \Z^d\setminus \{0\}$ and $j=0$ and therefore also for all $j\ne k$. 

    We observe now that there is a continuous function $\zeta:\R^d\ra[0,1]$ such that
    \[
        \zeta\vert_C \equiv 1, \quad \zeta\vert_{\R^d\setminus U} \equiv 0.
    \]
    Then the product $\zeta \tilde u$ is  supported on $C$ and it follows from \eqref{eqn_GGG} that 
    \begin{align*}
        \norm{\zeta \tilde u}_{L^l(\R^d)} &= \norm{\tilde u}_{L^l(C)} = \norm{\tilde u}_{L^l(U)} ,\qquad l\in \{p,q\}  ,
        \\ \norm{D^m(\zeta \tilde u)}_{L^r(\R^d)} &= \norm{D^m\tilde u}_{L^r(C)} = \norm{D^m\tilde u}_{L^r(U)}.
    \end{align*}
    An application of Lemma \ref{app-L-1} %{\color{red}with $\mathcal O = \R^d$} 
    in tandem with the above equalities yields then
    \begin{align}\label{ehn}
        \norm{\tilde u}_{L^p(C)} \lesssim_{(m,p,q,r)}  \, \norm{D^m\tilde u}_{L^r(C)}^a \norm{\tilde u}_{L^q(C)}^{1-a}.
    \end{align}

    To proceed, we make an argument that $\pi \colon U \to \pi(U)\subset \T^d$ is a bijection. While surjectivity is clear, we assume  for injectivity that there are two points $x,y\in C$ that get mapped to the same element of $ \T^d$ by the quotient map $\pi$. This would entail that $x- y =k $ for an integer $k$, violating the property  \eqref{eqn_GGG}. Therefore, $\pi$ is an inverse chart on $U$ preserving the Lebesgue measure. Let us also show that $\supp(u)\subset \pi(U)$, for which we consider $z \in \supp(u)\subset\T^d $. Then, any neighborhood of $x$ has positive $|u|\d x$-measure. Thus, also any neighborhood of $x$ with $\pi(x) = z$ has positive $|\tilde u| \d x$-measure. Therefore, we have $x\in \supp(\tilde u)$ which by \eqref{eqn_GG} implies that $x+ k \in C$ for some $k\in \Z^d$. But this yields the desired $\pi(y) = z$ for $y=x+k\in C\subset U$. Finally, by the preceding observations, we deduce that 
     \[
        \norm{\tilde u}_{L^p(U)}=\left(\int_U \abs{\tilde u(x)}^p \, \d x\right)^{1/p} = \left(\int_{\pi(U)} \abs{ u(x)}^p \, \d x\right)^{1/p} = \norm{u}_{L^p(\T^d)}.
    \]
    Using analogous identities for $ \| u\|_{L^q(\T^d)}$ and $\|D^m u\|_{L^r(\T^d)}$, the desired \eqref{eq-GN-1} follows from \eqref{ehn}.
\end{proof}

When applying the preceding estimate, we use the following criterion for a function to have compact effective support. 
\begin{lemma}
    \label{app-L-3}
     Let $\tilde{\phi}:\R^d\ra \R^d$ be a homeomorphism satisfying 
     \begin{align}\label{eqn_ravensburg}
     \tilde\phi(x+k ) = \tilde\phi(x )+k, \qquad k \in \Z^d ,\, x\in \R^d,
     \end{align}
     and  $ \phi:\T^d \ra \T^d$ the corresponding homeomorphism on the torus. Then, a measurable  $u:\T^d\ra \R$ has compact effective support iff $u\circ  \phi\inv$ has this property.  
\end{lemma}

\begin{proof}
    Without loss of generality, we suppose again that $L=1$. It suffices to show one implication since  $u = (u\circ \phi\inv )\circ \phi $. To this end, we notice that
    \[
        \supp(u\circ  \phi\inv) =  \phi(\supp(u)),
    \]
    from which we obtain based on \eqref{eqn_ravensburg}
    \[
        \pi\inv(\supp(u\circ  \phi\inv) ) = \pi\inv ( \phi(\supp(u))) = \tilde \phi(\pi\inv(\supp(u))) = \bigcup_{k\in \Z^d} k+\tilde \phi(C). 
    \]
    Since $\tilde{\phi}$ is continuous, $\tilde\phi(C)$ is compact and it remains to check that the above union is disjoint. For this, suppose that $j,k \in \Z^d$ such that $j \neq k$. Using once more \eqref{eqn_ravensburg} and that $\tilde\phi $ is a bijection, we find 
    \begin{align*}
        \tilde\phi\inv\left( (j+\tilde\phi(C))\cap(k+\tilde\phi(C))\right) & = \tilde\phi\inv(j+\tilde\phi(C)) \cap \tilde\phi\inv(k+\tilde\phi(C))
        \\& =(j+C)\cap (k+C)
        \\& = \emptyset,
    \end{align*}
    which finishes the proof.
\end{proof}

\noindent
\textbf{Acknowledgements.}
J.U.\ has been supported by the Graduiertenkolleg 2339 IntComSin
”Interfaces, Complex Structures, and Singular Limits” of the Deutsche Forschungsgemeinschaft (DFG, German Research Foundation) with Project-ID 321821685. The support
is gratefully acknowledged. M.S. would like to thank the FAU Erlangen-Nürnberg and the Graduiertenkolleg 2339 for their hospitality during several stays in Erlangen in order to work on this project. Both authors would like to thank G. Gr\"un for many lucrative discussions about the topic of the manuscript. 

\vspace{.2cm}
\noindent
\textbf{Data availability.} This manuscript has no associated data.

\vspace{.2cm}
\noindent
\textbf{Declaration – Conflict of interest.} The authors have no conflict of interest.

\vspace{.2cm}
\noindent
\textbf{AI disclosure statement.}
The large language model Gemini 3.1 Pro by Google was used to check the explicit integral calculation in the proof of Corollary \ref{col:upper}.
The authors take full responsibility for the contents of the manuscript.

\bibliographystyle{amsplain}
\bibliography{SPME_bib_linear}

\end{document}